\documentclass[12pt]{amsart}
\usepackage{amsfonts,amssymb,amscd,amsmath,enumerate,verbatim}
\usepackage[latin1]{inputenc}
\usepackage{amscd}
\usepackage{latexsym}
\usepackage{ytableau}
\usepackage{tikz}
\usepackage{tikz-cd}
\usepackage{mathrsfs}
\usepackage{enumitem}
\usepackage{hyperref}
\usepackage{cleveref}
\ytableausetup{boxsize=0.5em}

\usepackage{graphicx}

\usepackage{mathptmx}

\def\NZQ{\mathbb}               

\def\ZZ{{\NZQ Z}}
\def\RR{{\NZQ R}}
\def\CC{{\NZQ C}}

\def\frk{\mathfrak}               

\def\Sf{{\frk S}}

\def\Phi{{\frk N}}
\def\g{\gamma}

\def\opn#1#2{\def#1{\operatorname{#2}}} 
\opn\ini{in} \opn\sgn{sgn} \opn\GL{GL} \opn\Im{Im} \opn\GL{GL} \opn\EE{Ehr} \opn\Id{Id}
\opn\SYT{SYT} \opn\des{des} \opn\auto{Aut} \opn\trace{Tr}
\opn\gr{gr}

\newtheorem{Theorem}{Theorem}[section]
\newtheorem{TheoremM}{Theorem}

\newtheorem{Lemma}[Theorem]{Lemma}
\newtheorem{Corollary}[Theorem]{Corollary}

\theoremstyle{definition}

\newtheorem{Definition}[Theorem]{Definition}

\newtheorem{Question}[Theorem]{Question}

\theoremstyle{remark}
\newtheorem{Remark}[Theorem]{Remark}

\let\epsilon\varepsilon
\let\phi=\varphi
\let\kappa=\varkappa

\begin{document}

\title{Order polytopes of graded posets are gamma-effective}

\author{Alessio D'Al\`i and Akihiro Higashitani}

\address{Dipartimento di Matematica, Politecnico di Milano, Italy}
\email{alessio.dali@polimi.it}

\address{Akihiro Higashitani,
Department of Pure and Applied Mathematics,
Graduate School of Information Science and Technology,
Osaka University, 
Suita, Osaka 565-0871, Japan}
\email{higashitani@ist.osaka-u.ac.jp}


\begin{abstract}
Order polytopes of posets have been a very rich topic at the crossroads between combinatorics and discrete geometry since their definition by Stanley in 1986. Among other notable results, order polytopes of graded posets are known to be $\gamma$-non\-neg\-a\-tive by work of Br\"and\'en, who introduced the concept of sign-graded poset in the process.

In the present paper we are interested in proving an equivariant version of Br\"and\'en's result, using the tools of equivariant Ehrhart theory (introduced by Stapledon in 2011). Namely, we prove that order polytopes of graded posets are always $\gamma$-effective, i.e., that the $\gamma$-polynomial associated with the equivariant $h^*$-polynomial of the order polytope of any graded poset has coefficients consisting of actual characters. To reach this goal, we develop a theory of order polytopes of sign-graded posets, and find a formula to express the numerator of the equivariant Ehrhart series of such an object in terms of the saturations (\`a la Br\"and\'en) of the given sign-graded poset.
\end{abstract}

\maketitle

\section{Introduction}
In the present work we investigate the interaction between Ehrhart theory, group actions, and the combinatorics of (sign-)graded posets. More precisely, it is our main goal to prove that order polytopes of graded posets are $\gamma$-effective, thus generalizing to the equivariant setting a celebrated result of Br\"and\'en \cite{branden04}. In order to make sense of such a statement, we first need to build some vocabulary.

\subsection{Ehrhart theory, \texorpdfstring{$\gamma$}{gamma}-nonnegativity, and their equivariant counterparts}

Let $M \cong \mathbb{Z}^d$ be a $d$-dimensional lattice, let $M_{\mathbb{R}} = M \otimes_{\ZZ} \mathbb{R}$, and let $P \subset M_{\RR}$ be a full-dimensional $M$-lattice polytope, i.e., the convex hull of finitely many points in $M$. One can then consider how many lattice points are contained in $P$ and in each of its dilations $mP=\{m \alpha \mid \alpha \in P\}$, where $m \in \mathbb{Z}_{>0}$. It is customary to collect such data into the formal power series 
\begin{equation} \label{eq:intro Ehrhart}
\EE(P;t):=1+\sum_{m \geq 1}|mP \cap\ZZ^d|t^m = \frac{h_0^*+h_1^*t+\cdots+h_{s}^*t^{s}}{(1-t)^{\dim P+1}},
\end{equation} 
where $\dim P$ is the dimension of the affine hull of $P$. The formal series $\EE(P;t)$ is called the \textit{Ehrhart series} of $P$, and the polynomial appearing in the numerator of \eqref{eq:intro Ehrhart} is called the \textit{$h^*$-polynomial} of $P$, denoted by $h^*(P;t)$. The degree $s$ of $h^*(P;t)$ is called the \emph{degree} of the polytope $P$. For an introduction to Ehrhart theory, see \cite{beck-robins}.

Next, let us recall the notion of $\gamma$-nonnegativity for a palindromic polynomial with real coefficients.
Let $f(t)=\sum_{i=0}^sa_it^i$ be a polynomial of degree $s$ with $a_i \in \RR$. 
We say that $f(t)$ is \textit{palindromic} if $a_i=a_{s-i}$ holds for $i=0,1,\ldots,s$. 

If $f(t)$ is palindromic, then there exist $\gamma_0,\gamma_1,\ldots,\gamma_{\lfloor s/2 \rfloor} \in \RR$ such that 
\[f(t)=\sum_{i=0}^{\lfloor s/2 \rfloor}\gamma_i t^i(1+t)^{s-2i}.\]
We say that $f(t)$ is \textit{$\gamma$-nonnegative} (or sometimes \textit{$\gamma$-positive}) if $\gamma_i \geq 0$ for each $i$. The $\gamma$-nonnegativity property for a given palindromic polynomial $f(t)$ arises quite frequently in cases of combinatorial interest; moreover, it provides a middle ground between other well-studied properties. Indeed, if $f(t)$ is real-rooted then it must also be $\gamma$-nonnegative, and if it is $\gamma$-nonnegative then its coefficients $a_0, a_1, \ldots, a_s$ form a unimodal sequence. A very intriguing problem regarding $\gamma$-nonnegativity is \emph{Gal's conjecture}: is the $h$-polynomial of a flag homology sphere always $\gamma$-nonnegative? For a survey on $\gamma$-positivity including a discussion of Gal's conjecture we refer the reader to, e.g., \cite{athanasiadis-gamma}.

One might arrange the information of the $\gamma_i$'s into a new polynomial having these as coefficients, namely,
\[\gamma(f;t):=\sum_{i=0}^{\lfloor s/2 \rfloor} \gamma_it^i.\]
We call $\gamma(f;t)$ the \textit{$\gamma$-polynomial} of $f(t)$. A property of the $\gamma$-polynomial that will come in handy in what follows is that, if the palindromic polynomial $f(t)$ can be written as a product of two palindromic polynomials $f_1(t)$ and $f_2(t)$, then so is its $\gamma$-polynomial; more precisely, $f(t)=f_1(t)f_2(t)$ implies $\gamma(f;t)=\gamma(f_1;t)\gamma(f_2;t)$. 

Let us now go back to lattice polytopes. A characterization of those polytopes whose $h^*$-polynomial is palindromic is available; indeed, if $P$ is a full-dimensional $M$-lattice polytope in $M_{\RR} \cong \RR^d$, then its $h^*$-polynomial is palindromic (of degree $s$) if and only if the $(d+1-s)$-th dilation of $P$ is a translate of a reflexive polytope (cf. \cite[Theorem 2.1]{stapledon11}). It is however unclear which polytopes admit a $\gamma$-nonnegative $h^*$-polynomial, and this has been a very active field of research in the last few years: for a recent survey, see \cite[Section 3.3]{ferroni-higashitani}.
 
We now wish to deal with a \emph{symmetric} version of the above setting; namely, we want to take into the picture the action of a finite group. In the context of lattice polytopes, this idea gives rise to the concept of \emph{equivariant Ehrhart series}, formalized by Stapledon in \cite{stapledon11}. See, e.g., \cite{dali-delucchi} and \cite{stapledon24} for some recent developments. 

Let $M\cong \ZZ^d$ be a lattice and let $G$ be a finite group acting affinely on $M$ (i.e., there is a group homomorphism $\rho\colon G \to \mathrm{Aff}(M)$). Consider the lattice $\widetilde{M} := M \oplus \ZZ$, and note that every affine transformation $\phi$ of $M$ can be made into a linear transformation $\widetilde\phi \in \GL(\widetilde M)$ by asking that $\widetilde\phi(x \oplus u) = \phi(x) \oplus u$. In particular, the affine $G$-representation $\rho\colon G \to \mathrm{Aff}(M)$ extends to a linear $G$-representation $\widetilde\rho\colon G \to \GL(\widetilde M)$. Tensoring with $\mathbb{K} \in \{\RR,\CC\}$, the $G$-representations $\rho$ and $\widetilde\rho$ extend to (real or complex) $G$-representations $\rho_{\mathbb{K}}\colon G \to \mathrm{Aff}(M_{\mathbb{K}})$ and $\widetilde\rho_{\mathbb{K}}\colon G \to \GL(\widetilde{M}_{\mathbb{K}})$, where $M_{\mathbb{K}} := M \otimes \mathbb{K}$ and $\widetilde{M}_{\mathbb{K}} := \widetilde{M} \otimes \mathbb{K}$. 

Let $P \subset M_{\RR}$ be an $M$-lattice polytope that is invariant by the affine $G$-action $\rho_{\RR}$. Then, embedding the polytope $P$ at height one in $\widetilde{M}_{\RR}$, we get that $P \oplus \{1\}$ is invariant by the linear $G$-action $\widetilde\rho_{\RR}$, and so is each dilation $mP \oplus \{m\}$, where $m \in \ZZ_{>0}$. In particular, for each $m$, every element of $G$ permutes the lattice points in $mP \oplus \{m\}$. Let $\chi_{mP} \in \mathrm{R}_G$ be the character of the associated complex permutation representation, where $\mathrm{R}_G$ denotes the complex ring of virtual characters of $G$ (see Subsection~\ref{sec:rep}). 

\begin{Definition}
With the notation above, the \textit{equivariant Ehrhart series} of $P$ with respect to the given group action is 
\[\EE(P,\rho;t):=1+\sum_{m \geq 1}\chi_{mP}t^m \in \mathrm{R}_G[[t]].\]
Mimicking the rational expression of the usual Ehrhart series, one can define $\mathfrak{h}^*(P,\rho;t)$ as 
\[\mathfrak{h}^*(P,\rho;t) := \EE(P,\rho;t) \cdot \det(\Id-\widetilde{\rho}_{\CC} t),\]
where $\det(\Id-\widetilde\rho_{\CC}t)$ is the polynomial in $\mathrm{R}_G[t]$ defined as $\sum_{i=0}^{d+1}(-1)^i\chi_{\wedge^i\widetilde\rho_{\CC}}t^i$. Note however that ${\mathfrak{h}^*}(P,\rho;t)$ is a power series in $\mathrm{R}_G[[t]]$ and not necessarily a polynomial. The power series ${\mathfrak{h}^*}(P,\rho;t)$ is called the \textit{equivariant $h^*$-series} of $P$, or the \textit{equivariant $h^*$-polynomial} when it is indeed a polynomial. Depending on the source, ${\mathfrak{h}^*}(P,\rho;t)$ is sometimes denoted by $\phi(t)$ or $H^*(t)$. We say that $\mathfrak{h}^*(P,\rho;t)$ is \emph{effective} if so are all of its coefficients, i.e., if such coefficients are actual $G$-characters (see \Cref{sec:rep} for more details).
\end{Definition}

The equivariant Ehrhart series of $P$ ``remembers the symmetries'' of the given polytope in the following sense: by evaluating $\EE(P,\rho;t)$ at $g \in G$, one gets the Ehrhart series of the rational polytope $P^g:=\{x \in P \mid g \cdot x = x\}$, i.e., $\EE(P,\rho;t)(g)=\EE(P^g;t)$ \cite[Lemma 5.2]{stapledon11}. 
In particular, one recovers the classical Ehrhart series of $P$ by evaluating the equivariant Ehrhart series $\EE(P,\rho;t)$ at the unit element $e \in G$. 

Since its definition, the equivariant $h^*$-series ${\mathfrak{h}^*}(P,\rho;t)$ has been studied by several researchers: see, e.g., \cite{stapledon11, ASV, CHK, elia-kim-supina, stapledon24, dali-delucchi}. A driving topic in equivariant Ehrhart theory is the so-called \textit{effectiveness conjecture}\footnote{Note that the original conjecture claimed the equivalence of \emph{three} conditions, but it has since been shown by Santos and Stapledon that the third one (namely, the existence of a $G$-invariant nondegenerate hypersurface for the toric variety of $P$) is strictly stronger than $h^*$-effectiveness and polynomiality. For more details and an explicit counterexample, see \cite[Subsection 3.4.3]{elia-kim-supina}.}  (\cite[Conjecture 12.1]{stapledon11}) by Stapledon: 
\begin{center}${\mathfrak{h}^*}(P,\rho;t)$ is a polynomial if and only if ${\mathfrak{h}^*}(P,\rho;t)$ is effective.\end{center} 
The ``if'' part of the conjecture is easily shown to be true, i.e., one has that ${\mathfrak{h}^*}(P,\rho;t)$ is a polynomial if it is effective. The ``only if'' part of the conjecture is known to hold for some specific classes of polytopes and group actions, including simplices \cite[Proposition 6.1]{stapledon11}, hypercubes \cite[Section 9]{stapledon11}, permutahedra and hypersimplices under the action of the symmetric group (respectively \cite[Theorem 1.3]{ASV} and \cite[Theorem 3.60]{elia-kim-supina}), lattice polytopes admitting a $G$-invariant lattice triangulation \cite[Theorem 1.4]{stapledon24}, and so on. 
The effectiveness conjecture is still open in general, while a counterexample is known if one extends the conjecture from lattice polytopes to pseudo-integral polytopes: see \cite{CHK} for more details.

Let us now assume that the equivariant $h^*$-series is a palindromic polynomial in $\mathrm{R}_G[t]$. Then one can define the associated (equivariant) $\gamma$-polynomial, and the following question naturally arises: 
\begin{Question}\label{qu:intro}
Let $G$ be a finite group, let $M$ be a lattice, let $\rho\colon G \to \mathrm{Aff}(M)$ be an affine representation, and let $P \subset M_{\mathbb{R}}$ be an $M$-lattice polytope left invariant by $\rho_{\mathbb{R}} = \rho \otimes \RR$. Assume further that $h^*(P;t)$ is palindromic and $\gamma$-nonnegative and that $\mathfrak{h}^*(P,\rho;t)$ is a polynomial. 
Is ${\mathfrak{h}^*}(P,\rho;t)$ $\gamma$-effective, i.e., are all coefficients of the equivariant $\gamma$-polynomial of $P$ effective? 
\end{Question}

In general, the answer to \Cref{qu:intro} is negative; for an example, think of the standard $3$-dimensional cross-polytope $\diamond_3 := \mathrm{conv}\{\pm \mathbf{e}_1, \pm \mathbf{e}_2, \pm \mathbf{e}_3\}$ under the coordinate-permuting action $\rho$ of $\Sf_3$. The $h^*$-polynomial of $\diamond_3$ is well-known to be $(1+t)^3$, which implies that $h^*(\diamond_3;t)$ is $\gamma$-nonnegative. However, computing the equivariant $h^*$-series -- in this case, a polynomial -- of $\diamond_3$ one gets that $\mathfrak{h}^*(\diamond_3,\rho;t) = 1 + (1+\chi^{\tiny{\ydiagram{2,1}}})t + (1+\chi^{\tiny{\ydiagram{2,1}}})t^2 + t^3$, where $\chi^{\tiny{\ydiagram{2,1}}}$ is the character of the standard representation of $\mathfrak{S}_3$ (see Table~\ref{tab:S_3} in \Cref{sec:equivariant gamma-polynomials}), and thus the equivariant $\gamma$-polynomial is $\boldsymbol{\gamma}(\diamond_3,\rho;t) = 1 + (-2 + \chi^{\tiny{\ydiagram{2,1}}})t$.

\subsection{Main results}

The present paper deals with order polytopes of posets and their generalization to sign-graded posets. If $P$ is a finite poset, the \emph{order polytope} of $P$ is \[\mathscr{O}(P) := \{f \colon P \to [0,1] \mid p \leq_P q \Rightarrow f(p) \geq f(q)\}.\]
Such an object was introduced by Stanley \cite{stanley-poset} and has since been central in Ehrhart theory. Say that the poset $P$ is \emph{graded} if every maximal saturated chain in $P$ has the same length. Hibi proved in \cite{hibi-distributive} that the $h^*$-polynomial of the order polytope of a graded poset is palindromic. A fundamental result on the $h^*$-polynomials of order polytopes, due to Br\"and\'en \cite{branden04}, is the following: 
\begin{Theorem}[{\cite[Theorem 4.2]{branden04}}]\label{thm:branden}
If $P$ is a graded poset, then the $h^*$-polynomial of $\mathscr{O}(P)$ is (palindromic and) $\gamma$-nonnegative. 
\end{Theorem}

In order to prove such a result, Br\"anden introduced \emph{sign-graded posets} as a generalization of graded posets: see \Cref{sec:sign-graded posets} for the relevant definitions. The main goal of this paper is to establish an equivariant version of Theorem~\ref{thm:branden}, thus providing a large class of polytopes for which the answer to \Cref{qu:intro} is affirmative. To achieve such a result, we need to associate with each sign-consistent poset $(P,\varepsilon)$ a convex object $\mathscr{O}(P,\varepsilon)$, see \Cref{sec:order polytopes of sign-graded posets}. With a slight abuse of notation, we will call $\mathscr{O}(P,\varepsilon)$ the order polytope of $(P,\varepsilon)$. Such an object is convex but usually not closed, and is obtained from the usual order polytope $\mathscr{O}(P)$ by removing a suitable set of facets.

The first main result of the present paper yields an expression of the equivariant $h^*$-series of $\mathscr{O}(P,\varepsilon)$, that we denote by ${\mathfrak{h}^*}^G_{P,\varepsilon}(t)$, for an $\varepsilon$-consistent poset $P$. 
\begin{TheoremM}[{see Theorem~\ref{thm:main theorem on h^*}}] \label{thm:intro main theorem h^*}
Let $P$ be an $\varepsilon$-consistent poset and let $G$ be a subgroup of $\mathrm{Aut}(P, \varepsilon)$. Then
\begin{equation*}
{\mathfrak{h}^*}^G_{P,\varepsilon}(t) = \sum_{i=1}^k \mathrm{Ind}^G_{\mathrm{stab}_G(Q_i,\delta_i)}\mathrm{Res}^{\mathrm{Aut}(Q_i, \delta_i)}_{\mathrm{stab}_G(Q_i,\delta_i)}{\mathfrak{h}^*}^{\mathrm{Aut}(Q_i, \delta_i)}_{Q_i, \delta_i}(t),
\end{equation*}
where $(Q_1,\delta_1),\ldots, (Q_k,\delta_k)$ are representatives of the $G$-orbits of the set of all saturations of $(P,\varepsilon)$.
\end{TheoremM}

The second main result is the following: 
\begin{TheoremM}[{see Theorem~\ref{thm:gamma-effective}}] \label{thm:intro main theorem gamma}
Let $P$ be a graded poset and let $G$ be a subgroup of $\mathrm{Aut}(P)$. 
Then ${\mathfrak{h}^*}^G_{P}(t)$ is $\gamma$-effective. 
\end{TheoremM}

Note that, even though the statement of \Cref{thm:intro main theorem gamma} does not involve sign-graded posets, its proof does.

\subsection{Structure of this paper} 
Here is a brief description of the contents of the present paper. 
In \Cref{sec:sign-graded posets}, we recall what a sign-graded poset is and investigate its relations to quotient posets and saturations.  
In \Cref{sec:order polytopes of sign-graded posets}, we associate a convex object $\mathscr{O}(P,\varepsilon)$ with each $\varepsilon$-consistent poset $(P,\varepsilon)$, study its equivariant Ehrhart series and prove \Cref{thm:intro main theorem h^*}. 
In \Cref{sec:equivariant gamma-polynomials}, we introduce equivariant $\gamma$-polynomials of graded posets and prove \Cref{thm:intro main theorem gamma}.
In \Cref{sec:example}, we discuss in detail a concrete example. 

\subsection*{Acknowledgements}
AD is a member of INdAM--GNSAGA and has been partially supported by the PRIN 2020355B8Y grant ``Squarefree Gr\"obner degenerations, special varieties and related topics'' and by the PRIN 2022K48YYP grant ``Unirationality, Hilbert schemes, and singularities''. AH is partially supported by JSPS KAKENHI Grant Number JP24K00521 and JP21KK0043. The idea for the present paper originated in February 2024, when AH stayed at Politecnico di Milano for a research visit funded by the PRIN and JSPS grants mentioned above.

The authors are grateful to Matt Beck and Emanuele Delucchi for useful discussions on the topics of the present paper.

\section{Sign-graded posets and their automorphisms} \label{sec:sign-graded posets}
This section builds on the idea of a sign-graded poset, developed by Br\"and\'en in his paper \cite{branden04}, and complements it by considering appropriate notions of morphisms and quotients.

\subsection{Notation and first notions}
Let $P$ be a finite poset, let $E(P)$ be the set of covering relations of $P$ (i.e., those pairs $p \lessdot q$ of elements of $P$ such that $p < q$ and there exists no $z \in P$ such that $p < z < q$), and let $\varepsilon\colon E(P) \to \{-1,1\}$ (we will call such an $\varepsilon$ an \emph{edge labeling} of $P$).
\begin{Definition}
\phantom{ }
\begin{itemize}
    \item We define $\mathrm{Aut}(P,\varepsilon) := \{\psi \in \mathrm{Aut}(P) \mid \varepsilon(\psi(p) \lessdot \psi(q)) = \varepsilon(p \lessdot q) \textrm{ for all }p \lessdot q \in E(P)\}.$ Note that $\mathrm{Aut}(P,\varepsilon)$ is a subgroup of $\mathrm{Aut}(P)$.
    \item If for every maximal saturated chain $\mathcal{C}\colon p_0 \lessdot p_1 \lessdot \ldots \lessdot p_m$ of $P$, the quantity $\varepsilon(\mathcal{C}) := \sum_{i=0}^{m-1}\varepsilon(p_i \lessdot p_{i+1})$ is constant (note that the number $m$ may vary), we say that the pair $(P,\varepsilon)$ is a \emph{sign-graded poset}\footnote{The definition in the original paper by Br\"and\'en is slightly different, as it involves a \emph{vertex} labeling $\omega$ inducing the edge labeling $\varepsilon$. See also \Cref{rem:vertex labeling}.}, or that the poset $P$ is \emph{$\varepsilon$-graded}. In such case, we will write $r_P(\varepsilon)$ to denote the value of $\varepsilon(\mathcal{C})$ for any maximal saturated chain $\mathcal{C}$.
    \item If, for every $y \in P$, one has that the principal order ideal $P_{\leq y} = \{x \in P \mid x \leq y\}$ is $\varepsilon_y$-graded (where $\varepsilon_y$ is the restriction of $\varepsilon$ to $E(P_{\leq y})$), we say that $P$ is \emph{$\varepsilon$-consistent}.
    \item If $P$ is $\varepsilon$-consistent and $x \in P$, one defines the \emph{rank} of $x$ to be $\rho(x) := \varepsilon(\mathcal{C})$, where $\mathcal{C}$ is any maximal saturated chain of $P_{\leq x}$. This assignment produces a rank function $\rho\colon P \to \mathbb{Z}$.
\end{itemize}
\end{Definition}

A finite poset $P$ is graded (in the sense of Stanley: see, e.g., \cite{stanley-ec1}) if all maximal saturated chains have the same number of edges. In our language, $P$ is graded precisely when it can be given a $\mathbf{1}$-graded structure, i.e., when $P$ is sign-graded with respect to the edge labeling $\mathbf{1}$ that assigns a 1 to each element of $E(P)$. When this is the case, $\mathrm{Aut}(P,\mathbf{1})$ coincides with $\mathrm{Aut}(P)$. This reflects the known fact that automorphisms of a graded poset preserve its rank. In what follows, if $P$ is $\mathbf{1}$-graded, we will sometimes drop the ``$\mathbf{1}$'' from our notation.

\begin{Remark} \label{rem:vertex labeling}
As we already noted, Br\"and\'en gave a slightly different definition of sign-graded poset, focusing on the concept of a vertex labeling $\omega\colon P \to \{1, 2, 3, \ldots, |P|\}$. This is related to the combinatorial topic of $(P,\omega)$-partitions. We won't need this viewpoint for the present paper, but we wish to comment briefly on the connection between vertex and edge labelings. Given any vertex labeling $\omega$, Br\"and\'en describes how to associate with it an induced edge labeling $\varepsilon$: given $p \lessdot q \in E(P)$, set $\varepsilon(p \lessdot q) = 1$ (respectively, $-1$) if $\omega(p) < \omega(q)$ (respectively, $\omega(p) > \omega(q)$). If instead $\varepsilon$ is an edge labeling and $P$ is $\varepsilon$-consistent, we can easily find a vertex labeling $\omega$ inducing the edge labeling $\varepsilon$: indeed, it is enough to group the elements of $P$ by rank and choose $\omega$ so that $\rho(p) < \rho(q)$ implies that $\omega(p) < \omega(q)$.
\end{Remark}

We collect here some observations for later use.

\begin{itemize}
    \item Being $\varepsilon$-graded implies being $\varepsilon$-consistent; in fact, $P$ is $\varepsilon$-graded if and only if it is $\varepsilon$-consistent and the rank function $\rho$ assumes the same value on all maximal elements of $P$.
    \item If $P$ is $\varepsilon$-consistent and $p \lessdot p' \in E(P)$, then $\varepsilon(p \lessdot p') = \rho(p') - \rho(p)$. In particular, it follows that $|\rho(p') - \rho(p)| = 1$.
    \item If $P$ is $\varepsilon$-consistent and $\psi \in \mathrm{Aut}(P,\varepsilon)$, then $\rho(p) = \rho(\psi(p))$ for any $p \in P$; in other words, $\psi$ preserves the rank function of $(P,\varepsilon)$.
\end{itemize}

\begin{Definition}
Let $P$ and $Q$ be finite posets. The \textit{ordinal sum} $P \oplus Q$ is the poset obtained by equipping the disjoint union of $P$ and $Q$ with the partial order defined by 
\[x <_{P \oplus Q} y \overset{\text{def}}{\Longleftrightarrow} x < y \text{ in }P, \text{ or }x < y \text{ in }Q, \text{ or }x \in P, y \in Q.\]
Now let $\varepsilon\colon E(P) \to \{1,-1\}$ and $\delta\colon E(Q) \to \{1,-1\}$. We define two labelings $\varepsilon \oplus_{\pm 1} \delta$ of $E(P \oplus Q)$ as follows: 
\[
(\varepsilon \oplus_1 \delta)(x \lessdot y)=\begin{cases}
\varepsilon(x \lessdot y) &\text{ if }x \lessdot y \in E(P), \\
\delta(x \lessdot y) &\text{ if }x \lessdot y \in E(Q), \\
1 &\text{ if }x \in P, \; y \in Q, 
\end{cases}\]
and
\[(\varepsilon \oplus_{-1} \delta)(x\lessdot y)=\begin{cases}
\varepsilon(x \lessdot y) &\text{ if }x \lessdot y \in E(P), \\
\delta(x \lessdot y) &\text{ if }x \lessdot y \in E(Q), \\
-1 &\text{ if }x \in P, \; y \in Q. 
\end{cases}\]

We will sometimes write $P \oplus_{1}Q$ (respectively, $P \oplus_{-1}Q$) to denote the ordinal sum $P \oplus Q$ equipped with the edge labeling $\varepsilon \oplus_{1} \delta$ (respectively, $\varepsilon \oplus_{-1} \delta$). 
\end{Definition}

\begin{Remark} \label{rem:parity-graded}
Br\"and\'en noted in \cite[Theorem 2.5]{branden04} that a poset $P$ is $\varepsilon$-graded (respectively, $\varepsilon$-consistent) for some edge labeling $\varepsilon\colon E(P) \to \{1,-1\}$ if and only if $P$ is \emph{parity-graded} (respectively, \emph{parity-consistent}), i.e., all maximal saturated chains in $P$ have the same number of edges modulo $2$ (respectively, the order ideal $P_{\leq q}$ is parity-graded for every choice of $q \in P$). If $P$ is parity-graded, one can define the edge labeling $\varepsilon_{\mathrm{par}}$ by requiring that $\varepsilon_{\mathrm{par}}(p \lessdot q) = (-1)^{\ell(p)}$, where $\ell(p)$ is the number of edges of any saturated chain starting at a minimal element of $P$ and ending at $p$. The rank function $\rho_{\mathrm{par}}$ associated with $\varepsilon_{\mathrm{par}}$ takes values in $\{0,1\}$. 
Note that the edge labeling $\varepsilon_{\mathrm{par}}$ corresponds to a canonical vertex labeling (see \cite{branden04}). 

Moreover, one has that $\mathrm{Aut}(P,\varepsilon_{\mathrm{par}}) = \mathrm{Aut}(P)$; indeed, if $\varphi \in \mathrm{Aut}(P)$, then $\varphi$ preserves the length of any maximal saturated chain, and thus for every $p \lessdot q \in E(P)$ one has that \[\varepsilon_{\mathrm{par}}(p\lessdot q) = (-1)^{\ell(p)} = (-1)^{\ell(\varphi(p))} = \varepsilon_{\mathrm{par}}(\varphi(p) \lessdot \varphi(q)).\]
As a consequence, if $P$ is $\varepsilon$-graded, then any element of $\mathrm{Aut}(P,\varepsilon)$ also preserves the parity-grading $\varepsilon_{\mathrm{par}}$.
\end{Remark}

\subsection{Quotient posets}

We begin this subsection by showing that the $\varepsilon$-consistent and $\varepsilon$-graded properties are transferred when passing from $P$ to its quotient by a subgroup of $\mathrm{Aut}(P,\varepsilon)$.

\begin{Lemma} \label{lem:sign-graded quotient}
Let $P$ be an $\varepsilon$-consistent (respectively, $\varepsilon$-graded) poset and let $G$ be a subgroup of $\mathrm{Aut}(P,\varepsilon)$. For any covering relation $\mathcal{O} \lessdot \mathcal{O'}$ in the quotient poset $P/G$, define \[\overline{\varepsilon}(\mathcal{O} \lessdot \mathcal{O}') := \varepsilon(p \lessdot p'),\] where $p$ and $p'$ are any two elements of $P$ such that $p \in \mathcal{O}$, $p' \in \mathcal{O}'$ and $p \lessdot p'$. Then:
\begin{enumerate}[label=(\roman*)]
\item the quotient poset $P/G$ is $\overline\varepsilon$-consistent (respectively, $\overline\varepsilon$-graded);
\item for every $\mathcal{O} \in P/G$, one has that $\rho_{P/G}(\mathcal{O}) = \rho_P(p)$ for any choice of $p \in \mathcal{O}$.
\end{enumerate}
\end{Lemma}
\begin{proof}
Let us first check that $\overline{\varepsilon}$ is well-defined. Assume by contradiction it's not. Then there exist $p, q \in \mathcal{O}$, $p', q' \in \mathcal{O'}$ such that $p \lessdot p'$, $q \lessdot q'$ and $\varepsilon(p \lessdot p') \neq \varepsilon(q \lessdot q')$. Since $p'$ and $q'$ lie in the same $G$-orbit, there exists $g \in G$ such that $q' = gp'$; since $g^{-1}$ preserves $\varepsilon$, one has that \[\varepsilon(q \lessdot q') = \varepsilon(q \lessdot gp') = \varepsilon(g^{-1}q \lessdot p') = \varepsilon(hp \lessdot p')\] for some $h \in G$. Thus, we have that $\varepsilon(p \lessdot p') \neq \varepsilon(hp \lessdot p')$. 

Let \[\mathcal{C}\colon u_0 \lessdot u_1 \lessdot \ldots \lessdot u_k = p \lessdot u_{k+1} = p' \lessdot u_{k+2} \lessdot \ldots \lessdot u_m\] be a maximal saturated chain of $P$ and consider \[\mathcal{C}'\colon v_0 \lessdot v_1 \lessdot \ldots \lessdot v_m,\] where \[v_i = \begin{cases}hu_i & 0 \leq i \leq k,\\ u_i & k+1 \leq i \leq m.\end{cases}\]
Then $\mathcal{C}'$ is also a maximal saturated chain of $P$ and one has that \[\begin{split}\sum_{i=0}^{m-1}\varepsilon(v_i \lessdot v_{i+1}) &= \sum_{i=0}^{k-1}\varepsilon(hu_i \lessdot hu_{i+1}) + \varepsilon(hp \lessdot p') + \sum_{i=k+1}^{m-1}\varepsilon(u_i \lessdot u_{i+1})\\ &= \sum_{i=0}^{k-1}\varepsilon(u_i \lessdot u_{i+1}) + \varepsilon(hp \lessdot p') + \sum_{i=k+1}^{m-1}\varepsilon(u_i \lessdot u_{i+1})\\ &\neq \sum_{i=0}^{m-1}\varepsilon(u_i \lessdot u_{i+1}),\end{split}\]
against the hypothesis that $P$ is $\varepsilon$-consistent.

Let us now prove claim (i) when $P$ is $\varepsilon$-graded (the proof for the $\varepsilon$-consistent case is analogous). Consider a maximal saturated chain $\widetilde{\mathcal{C}}\colon \mathcal{O}_0 \lessdot \mathcal{O}_1 \lessdot \ldots \lessdot \mathcal{O}_m$ in $P/G$. Since $\mathcal{O}_0 \lessdot \mathcal{O}_1$, we know that there exist $p_0 \in \mathcal{O}_0$ and $p_1 \in \mathcal{O}_1$ such that $p_0 \lessdot p_1$; by definition, $\overline{\varepsilon}(\mathcal{O}_0 \lessdot \mathcal{O}_1) = \varepsilon(p_0 \lessdot p_1)$. Now consider $\mathcal{O}_1 \lessdot \mathcal{O}_2$. There exist $q_1 \in \mathcal{O}_1$ and $q_2 \in \mathcal{O}_2$ such that $q_1 \lessdot q_2$; since $p_1$ and $q_1$ lie in the same $G$-orbit $\mathcal{O}_1$, then $q_1 = gp_1$ for some $g \in G$. Since $g^{-1} \in \mathrm{Aut}(P,\varepsilon)$, we have that $p_1 \lessdot g^{-1}q_2$ and $\overline{\varepsilon}(\mathcal{O}_1 \lessdot \mathcal{O}_2) = \varepsilon(q_1 \lessdot q_2) = \varepsilon(p_1 \lessdot p_2)$, where $p_2 = g^{-1}q_2$. Continuing in this fashion, we build a maximal saturated chain $\mathcal{C}\colon p_0 \lessdot p_1 \lessdot p_2 \lessdot \ldots \lessdot p_m$ such that $p_i \in \mathcal{O}_i$ for every $i$. Since $\overline{\varepsilon}(\widetilde{\mathcal{C}}) = \varepsilon(\mathcal{C})$, it follows that $P/G$ inherits the sign-graded structure from $P$, as desired. Claim (ii) also follows from this line of reasoning.
\end{proof}

\subsection{Saturations and their interaction with quotient posets}

We now introduce a notion, due to Br\"and\'en \cite[Section 3]{branden04}, that will be crucial for what follows.

\begin{Definition}
    Let $P$ be an $\varepsilon$-consistent poset and let $Q$ be a $\delta$-consistent poset. We say that $(Q,\delta)$ is a \emph{saturation} of $(P,\varepsilon)$ if:
    \begin{itemize}
    \item $Q = P$ as sets;
    \item for any $x, y \in Q$ such that $|\rho_Q(y) - \rho_Q(x)| = 1$, one has that $x$ and $y$ are comparable in $Q$ (``$(Q,\delta)$ is saturated'');
    \item whenever $x <_P y$, one also has that $x <_Q y$ (``$Q$ extends $P$'');
    \item $\rho_Q(x) = \rho_P(x)$ for any $x$.
    \end{itemize}
\end{Definition}

We now want to bring in a new ingredient; namely, the fact that acting on a saturation of $(P,\varepsilon)$ by an element $g \in \mathrm{Aut}(P,\varepsilon)$ produces another saturation of $(P,\varepsilon)$. As a consequence, $g$ permutes the set $\mathcal{Q}$ of saturations of $(P,\varepsilon)$.

If $(Q,\delta)$ is a saturation of $(P,\varepsilon)$ and $g \in \mathrm{Aut}(P,\varepsilon)$, we define $(gQ,g\delta)$ in the following way:
\begin{itemize}
\item $gQ = P$ as sets;
\item $x <_{gQ} y$ if $g^{-1}x <_Q g^{-1}y$ (and thus $x \lessdot_{gQ} y$ precisely when $g^{-1}x \lessdot_{Q} g^{-1}y$);
\item $g\delta(x \lessdot y) = \delta(g^{-1}x \lessdot g^{-1}y)$.
\end{itemize}

\begin{Lemma} \label{lem:G-action on saturations}
Let $P$ be an $\varepsilon$-consistent poset, let $(Q,\delta)$ be a saturation of $(P,\varepsilon)$ and let $g \in \mathrm{Aut}(P,\varepsilon)$. Then $(gQ,g\delta)$ is a saturation of $(P,\varepsilon)$. 
\end{Lemma}
\begin{proof}
The poset $gQ$ is $g\delta$-consistent since $Q$ is $\delta$-consistent; moreover, $gQ = P$ by definition. If $x, y \in gQ$ are such that $|\rho_{gQ}(y) - \rho_{gQ}(x)|= 1$, then $|\rho_Q(g^{-1}y) - 
\rho_Q(g^{-1}x)| = 1$, which implies that $g^{-1}y$ and $g^{-1}x$ are comparable in $Q$, since $Q$ is saturated. This means precisely that $y$ and $x$ are comparable in $gQ$.

Let us now see that $gQ$ extends $P$. If $x <_P y$, then $g^{-1}x <_P g^{-1}y$ since $g^{-1}$ is an automorphism of $P$; since $Q$ extends $P$, one then has that $g^{-1}x <_Q g^{-1}y$ and thus $x <_{gQ} y$.

Finally, let us prove that $\rho_{gQ}(x) = \rho_P(x)$ for any $x$. By definition, $\rho_{gQ}(x) = \rho_Q(g^{-1}x)$, which in turn equals $\rho_P(g^{-1}x)$ since $(Q,\delta)$ is a saturation of $(P,\varepsilon)$. Since $g^{-1} \in \mathrm{Aut}(P,\varepsilon)$, one also has that $\rho_P(g^{-1}x) = \rho_P(x)$, which concludes the proof.
\end{proof}

One might wonder which elements of $\mathrm{Aut}(P,\varepsilon)$ leave $(Q,\delta)$ fixed, i.e., who the stabilizer of $(Q,\delta)$ is. The stabilizer $\mathrm{stab}_G(Q,\delta)$ turns out to consist precisely of the elements of $\mathrm{Aut}(P,\varepsilon)$ that also belong to $\mathrm{Aut}(Q,\delta)$.

\begin{Lemma} \label{lem:stabilizer}
    Let $P$ be an $\varepsilon$-consistent poset, let $G$ be a subgroup of $\mathrm{Aut}(P,\varepsilon)$ and consider the action of $G$ on the set $\mathcal{Q}$ of saturations of $(P,\varepsilon)$. Then, for any $(Q,\delta) \in \mathcal{Q}$, one has that 
    \[\mathrm{stab}_G(Q,\delta) = G \cap \mathrm{Aut}(Q,\delta).\]
\end{Lemma}
\begin{proof}
    We are looking for those elements $g \in G$ such that $(gQ,g\delta) = (Q,\delta)$. One has by definition that $gQ = Q$ as sets. As regards the poset structure, one has that $(gQ, <_{gQ}) = (Q,<_Q)$ if and only if $g \in G \cap \mathrm{Aut}(Q)$. Finally, if $g \in G \cap \mathrm{Aut}(Q)$, one has that $g\delta = \delta$ if and only if $g \in \mathrm{Aut}(Q,\delta)$.
\end{proof}
    
\begin{Remark} \label{rem:nice saturations}
     By \cite[Proposition 3.4]{branden04} one has that, if $P$ is parity-graded, each saturation of $P$ is of the form $A_0 \oplus_{1} A_1 \oplus_{-1} \ldots \oplus_{\pm 1} A_k$, where each $A_i$ is an antichain, and is hence parity-graded. Denote such a saturation by $(Q,\varepsilon_{\mathrm{par}})$. The underlying poset $Q$, being the ordinal sum $A_0 \oplus A_1 \oplus \ldots \oplus A_k$, is also $\mathbf{1}$-graded, and one has that \[\mathrm{Aut}(Q,\varepsilon_{\mathrm{par}}) = \mathrm{Aut}(Q,\mathbf{1}) = \mathrm{Aut}(Q) = \Sf_{|A_0|} \times \Sf_{|A_1|} \times \ldots \times \Sf_{|A_k|}.\]
\end{Remark}

\begin{Lemma} \label{lem:all covering relations}
Let $Q$ be a $\delta$-consistent saturated poset and let $G$ be a subgroup of $\mathrm{Aut}(Q,\delta)$. Then:
\begin{enumerate}[label=(\roman*)]
\item if $q \lessdot q'$, then $gq \lessdot hq'$ for all $g,h\in G$;
\item if $q \lessdot q'$, then $\delta(q \lessdot q') = \delta(gq \lessdot hq')$ for all $g,h\in G$;
\item if $q \leq q'$, then $gq \leq hq'$ for all $g,h \in G$.
\end{enumerate}
\end{Lemma}
\begin{proof}
    Claim (iii) readily follows from claim (i).
    
    To prove claim (i), it is enough to show that, for all $q \lessdot q' \in E(Q)$ and $g \in G$, one has that $q \lessdot gq'$. (Indeed, if $h, h' \in G$, then $hq \lessdot h'q'$ if and only if $q \lessdot h^{-1}h'q'$.)

    If $q \lessdot q'$, then $|\rho(q')-\rho(q)| = 1$. Since $g \in \mathrm{Aut}(Q,\delta)$, one has that $\rho(gq') = \rho(q')$ and thus $|\rho(gq') - \rho(q)| = 1$. Since $Q$ is saturated, it follows that $gq'$ and $q$ must be comparable.
    
    If $gq' \leq q$, then $gq' \leq q'$; repeated applications of $g$ then yield that $q' \geq gq' \geq g^2q' \geq \ldots$. Since $\mathrm{Aut}(Q,\delta)$ is a finite group, there exists $N$ such that $g^N = \mathrm{id}_G$. Hence, $q' \geq gq' \geq g^Nq' = q'$ and so $q' = gq'$, which in particular implies that $q \lessdot gq'$.

    Assume now that $q < gq'$. Then there exists $r \in Q$ such that $q \lessdot r \leq gq'$. Since $q \lessdot r$ and $g$ preserves the rank function, one has that $1 = |\rho(r)-\rho(q)| = |\rho(r)-\rho(gq)|$. Since $Q$ is saturated, it follows that $r$ and $gq$ must be comparable. If $r \leq gq$, then $q \leq gq$ and reasoning as above yields that $q = gq$, and thus $q = gq \lessdot gq'$ as desired. If instead $gq < r (\leq gq')$, then the fact that $gq \lessdot gq'$ implies that $r = gq'$ and thus $q \lessdot gq'$.

    Finally, let us prove claim (ii). Let $q \lessdot q' \in E(Q)$ and let $g, h \in G$. Since $gq \lessdot hq'$ from claim (i), it follows that $\delta(gq \lessdot hq') = \rho(hq') - \rho(gq) = \rho(q') - \rho(q) = \delta(q \lessdot q')$.
\end{proof}

\begin{Corollary} \label{cor:quotient of a saturated poset}
    Let $Q$ be a $\delta$-consistent saturated poset and let $G$ be a subgroup of $\mathrm{Aut}(Q,\delta)$. Build $(Q/G,\overline{\delta})$ as in \Cref{lem:sign-graded quotient}. Then, for any $\mathcal{O}, \mathcal{O'} \in Q/G$ and for any choice of representatives $q \in \mathcal{O}$, $q' \in \mathcal{O}'$, one has that:
    \begin{enumerate}[label=(\roman*)]
        \item $\mathcal{O} \lessdot_{Q/G} \mathcal{O'}$ if and only if $q \lessdot_{Q} q'$;
        \item when the equivalent conditions in (i) are met, then $\overline{\delta}(\mathcal{O} \lessdot \mathcal{O'}) = \delta(q \lessdot q')$.
    \end{enumerate}
\end{Corollary}
\begin{proof}
    Due to how $Q/G$ is defined, all we need to prove is claim (i). If $q \lessdot_Q q'$, then $Gq \lessdot_{Q/G} Gq'$ by definition of $Q/G$. If $\mathcal{O} \lessdot_{Q/G} \mathcal{O}'$, then by definition of $Q/G$ there exist representatives $r \in \mathcal{O}$ and $r' \in \mathcal{O'}$ such that $r \lessdot_Q r'$. \Cref{lem:all covering relations}(i) then yields the desired result. 
\end{proof}

\begin{Lemma} \label{lem:saturations of a quotient}
    Let $P$ be an $\varepsilon$-consistent poset and let $G$ be a subgroup of $\mathrm{Aut}(P,\varepsilon)$. Then there is a bijection between the set of all saturations of $(P/G, \overline\varepsilon)$ and the set of those saturations $(Q,\delta)$ of $(P,\varepsilon)$ such that $G \subseteq \mathrm{Aut}(Q,\delta)$.
\end{Lemma}
\begin{proof}
    \phantom{ }
    
    \textbf{Claim 1}: with each saturation $(\widetilde{Q}, \widetilde{\delta})$ of $(P/G, \overline\varepsilon)$ we can associate a saturation $(Q,\delta)$ of $(P, \varepsilon)$ such that $G \subseteq \mathrm{Aut}(Q,\delta)$.
    
    \emph{Proof of Claim 1.} Let $(\widetilde{Q}, \widetilde{\delta})$ be a saturation of $(P/G, \overline\varepsilon)$. Define a new poset $Q$ and a map $\delta\colon E(Q) \to \{-1,1\}$ as follows:
    \begin{itemize}
    \item $Q = P$ as sets;
    \item one has that $p \leq_Q p'$ if and only if $Gp \leq_{\widetilde{Q}} Gp'$ (note that implies that $p \lessdot_Q p'$ is a covering relation of $Q$ precisely when $Gp \lessdot_{\widetilde{Q}} Gp'$ is a covering relation of $\widetilde{Q}$);
    \item $\delta(p \lessdot_Q p') := \widetilde{\delta}(Gp \lessdot_{\widetilde{Q}} Gp')$.
    \end{itemize}
    Since $\delta(p \lessdot_Q p') = \widetilde{\delta}(Gp \lessdot_{\widetilde{Q}} Gp') = \delta(gp \lessdot_Q gp')$ for every covering relation $p \lessdot_Q p' \in E(Q)$ and for every $g \in G$, it follows immediately that $G$ is a subgroup of $\mathrm{Aut}(Q, \delta)$. Let us now prove that $(Q,\delta)$ is a saturation of $(P,\varepsilon)$. One has that $Q=P$ as sets by construction. To see that $Q$ extends $P$, consider $p, p' \in P$ such that $p \leq_P p'$. Then $Gp \leq_{P/G} Gp'$ and thus, since $\widetilde{Q}$ extends $P/G$, one has that $Gp \leq_{\widetilde{Q}} Gp'$, which in turn implies that $p \leq_Q p'$ by definition of $Q$. To see that $Q$ is $\delta$-consistent, pick $p \in Q$ and consider a maximal saturated chain $\mathcal{C}\colon q_0 \lessdot_Q q_1 \lessdot_Q \ldots \lessdot_Q q_m = p$ in $Q_{\leq p}$. Then $\widetilde{\mathcal{C}}\colon Gq_0 \lessdot_{\widetilde{Q}} Gq_1 \lessdot_{\widetilde{Q}} \ldots \lessdot_{\widetilde{Q}} Gq_m = Gp$ is a maximal saturated chain in $\widetilde{Q}_{\leq Gp}$, which is a $\widetilde{\delta}_{Gp}$-graded poset by assumption. Moreover, $\delta(\mathcal{C}) = \widetilde{\delta}(\widetilde{\mathcal{C}})$. This implies both that $Q$ is $\delta$-consistent and that $\rho_Q(p) = \rho_{\widetilde{Q}}(Gp)$ for every $p \in Q$. Since $(\widetilde{Q},\widetilde\delta)$ is a saturation of $(P/G,\overline\varepsilon)$, we have that $\rho_{\widetilde{Q}}(Gp) = \rho_{P/G}(Gp) = \rho_P(p)$, where the last equality is given by \Cref{lem:sign-graded quotient}(ii). Hence, $(Q,\delta)$ and $(P,\varepsilon)$ share the same rank function. Finally, consider $p, q \in Q$ such that $|\rho_Q(q)-\rho_Q(p)| = 1$. Then $|\rho_{\widetilde{Q}}(Gq)-\rho_{\widetilde{Q}}(Gp)| = 1$ and thus, since $(\widetilde{Q},\widetilde\delta)$ is a saturation of $(P/G,\overline\varepsilon)$, one has that $Gp$ and $Gq$ are comparable in $\widetilde{Q}$. But then $p$ and $q$ must be comparable in $Q$, as desired.

    \textbf{Claim 2}: with each saturation $(Q,\delta)$ of $(P, \varepsilon)$ such that $G \subseteq \mathrm{Aut}(Q,\delta)$ we can associate a saturation $(\widetilde{Q}, \widetilde{\delta})$ of $(P/G, \overline\varepsilon)$.

    \emph{Proof of Claim 2.} Let $(Q,\delta)$ be a saturation of $(P,\varepsilon)$ such that $G \subseteq \mathrm{Aut}(Q,\delta)$. By \Cref{lem:sign-graded quotient}, the quotient poset $(Q/G, \overline\delta)$ (which coincides with $P/G$ as a set) is $\overline\delta$-consistent and has the same rank function as $Q$. As a consequence, for every $\mathcal{O} \in Q/G$ and every $p \in \mathcal{O}$, one has that $\rho_{Q/G}(\mathcal{O}) = \rho_Q(p) = \rho_P(p) = \rho_{P/G}(\mathcal{O})$ (where we used that $(Q,\delta)$ is a saturation of $(P,\varepsilon)$ and the fact that passing from $P$ to $P/G$ preserves the rank function). To see that $Q/G$ extends $P/G$, consider $\mathcal{O}, \mathcal{O}' \in P/G$ such that $\mathcal{O} \leq_{P/G} \mathcal{O'}$. Then there exist $p \in \mathcal{O}$ and $p' \in \mathcal{O'}$ such that $p \leq_P p'$. Since $Q$ extends $P$, we also have that $p \leq_Q p'$, which in turn implies that $\mathcal{O} \leq_{Q/G} \mathcal{O'}$. Finally, consider $\mathcal{O}, \mathcal{O}' \in Q/G$ such that $|\rho_{Q/G}(\mathcal{O}')-\rho_{Q/G}(\mathcal{O})| = 1$. Then, for any $q \in \mathcal{O}$ and $q' \in \mathcal{O}'$, one has that $|\rho_{Q}(q')-\rho_{Q}(q)| = 1$. Since $(Q,\delta)$ is a saturation of $(P,\varepsilon)$, one has that $q$ and $q'$ are comparable in $Q$, which implies that $\mathcal{O}$ and $\mathcal{O'}$ must be comparable in $Q/G$, as desired.

    \textbf{Claim 3}: the maps described in Claims 1 and 2 yield a bijection.
    
    \emph{Proof of Claim 3.} Let $\lambda$ be the map described in Claim 1 and let $\pi$ be the map described in Claim 2 (and in \Cref{lem:sign-graded quotient}).

    Let $(Q, \delta)$ be a saturation of $(P, \varepsilon)$ such that $G \subseteq \mathrm{Aut}(Q, \delta)$. We wish to prove that $\lambda \circ \pi(Q,\delta) = (Q,\delta)$, i.e., that $\lambda(Q/G, \overline\delta) = (Q,\delta)$. Both $\lambda(Q/G)$ and $Q$ coincide with $P$ as sets. By \Cref{cor:quotient of a saturated poset}, one has that $q \lessdot_{Q} q'$ if and only if $Gq \lessdot_{Q/G} Gq'$, which is in turn equivalent to $q \lessdot_{\lambda(Q/G)} q'$ by the construction in Claim 1. Applying again \Cref{cor:quotient of a saturated poset} and Claim 1, it follows that $\delta_{\lambda(Q/G)}(q \lessdot q') = \delta_Q(q \lessdot q')$.

    Let $(\widetilde{Q}, \widetilde\delta)$ be a saturation of $(P/G,\overline\varepsilon)$. We wish to prove that $\pi \circ \lambda(\widetilde{Q},\widetilde\delta) = (\widetilde{Q},\widetilde\delta)$, i.e., that $\lambda(\widetilde{Q},\widetilde\delta)/G = (\widetilde{Q},\widetilde\delta)$. Both $\lambda(\widetilde{Q})/G$ and $\widetilde{Q}$ coincide with $P/G$ as sets.
    If $\mathcal{O} \lessdot_{\widetilde{Q}} \mathcal{O}'$, then for any representatives $q \in \mathcal{O}$ and $q' \in \mathcal{O'}$ one has that $q \lessdot_{\lambda(\widetilde{Q})} q'$, which in turn implies that $\mathcal{O} \lessdot_{\lambda(\widetilde{Q})/G} \mathcal{O}'$. Conversely, if $\mathcal{O} \lessdot_{\lambda(\widetilde{Q})/G} \mathcal{O}'$, then there exist representatives $q \in \mathcal{O}$, $q' \in \mathcal{O}'$ such that $q \lessdot_{\lambda(\widetilde{Q})} q'$, which implies that $\mathcal{O} \lessdot_{\widetilde{Q}} \mathcal{O}'$. One checks immediately that $\widetilde\delta(\mathcal{O} \lessdot \mathcal{O}') = \delta_{\lambda(\widetilde{Q})/G}(\mathcal{O} \lessdot \mathcal{O}')$, which concludes the proof.
\end{proof}

\section{Order polytopes of sign-graded posets \texorpdfstring{\\}{}and their equivariant Ehrhart series} \label{sec:order polytopes of sign-graded posets}

\subsection{Order polytopes of sign-graded posets}

First, we introduce $\mathscr{O}(P,\varepsilon)$ for an $\varepsilon$-con\-sis\-tent poset $P$. This is obtained from the usual order polytope of $P$ by removing a suitable set of facets, and coincides with it when $\varepsilon = \mathbf{1}$. For the precise definition, we need a new concept.

Let $P$ be an $\varepsilon$-consistent poset and let $p,p' \in P$ such that $p <_P p'$. We say that $p$ and $p'$ form an \emph{ascending pair} if, for every saturated chain $p=p_0 \lessdot p_1 \lessdot \ldots \lessdot p_m=p'$, one has that $\varepsilon(p_i \lessdot p_{i+1}) = 1$ for every $i$. When this is not the case, we say that $p$ and $p'$ form a \emph{nonascending pair}. 

\begin{Definition}
    Let $P$ be an $\varepsilon$-consistent poset. Then we define 
    $\mathscr{O}(P,\varepsilon)$ as the convex set consisting of those functions $f \colon P \to [0,1]$ such that \begin{equation} \label{eq:defining inequalities of an order polytope}
    \begin{split}
    &f(p) \geq f(q) \textrm{ if } p \leq_P q,\\ &f(p) > f(q) \textrm{ if } p <_P q \textrm{ and }p, q\textrm{ form a nonascending pair.}
    \end{split}
    \end{equation}

    We define the dimension of $\mathscr{O}(P,\varepsilon)$ to be the dimension of its affine hull. This turns out to be equal to $|P|$, since $\mathrm{relint}(\mathscr{O}(P)) \subseteq \mathscr{O}(P,\varepsilon) \subseteq \mathscr{O}(P)$ and any order polytope is known to be full-dimensional (here and in what follows, ``relint'' is the relative interior).
\end{Definition}

In what follows, if $X$ is any finite set and $x \in X$, we will write $\mathbf{e}_x^*$ to denote the function $X \to \mathbb{Z}$ such that $\mathbf{e}^*_x(y) = 1$ if $y=x$ and $0$ otherwise. The collection of all the possible $\mathbf{e}_x^*$ forms a basis for the lattice $\mathbb{Z}^X$ of all functions $X \to \mathbb{Z}$.

If $P$ is an $\varepsilon$-consistent poset and $G$ is a subgroup of $\mathrm{Aut}(P, \varepsilon)$ (which in turn is a subgroup of the group $\Sf_P$ of bijections of the set $P$), then $G$ acts linearly on the lattice $\ZZ^P$ by $g \mapsto (f \mapsto f \circ g^{-1})$, and such an action preserves $\mathscr{O}(P,\varepsilon)$ (see also \Cref{subsec:equivariant Ehrhart for order polytopes}). 
We will write $\mathscr{O}(P,\varepsilon)^G$ to denote the part of $\mathscr{O}(P,\varepsilon)$ that is fixed pointwise by each element of $G$. If $G = \langle g \rangle$, i.e., if $G$ is cyclically generated by $g$, one has that \[\mathscr{O}(P,\varepsilon)^{\langle g \rangle} = \mathscr{O}(P,\varepsilon)^g = \{f \in \mathscr{O}(P,\varepsilon) \mid g \cdot f = f\} = \{f \in \mathscr{O}(P,\varepsilon) \mid f \circ g^{-1} = f\}.\]

\begin{Lemma} \label{lem:order polytope of a quotient poset}
    Let $P$ be an $\varepsilon$-consistent poset and let $G$ be a subgroup of $\mathrm{Aut}(P,\varepsilon)$. Then there is a unimodular equivalence between $\mathscr{O}(P,\varepsilon)^G$ and $\mathscr{O}(P/G, \overline{\varepsilon})$.
\end{Lemma}

\begin{proof}
    To fix notation, we write down 
    \[\begin{split}\mathscr{O}(P,\varepsilon)^G = \{&f \colon P \to [0,1] \textrm{ such that }\\ &f(p) \geq f(q) \textrm{ if } p \leq_{P} q,\\ &f(p) > f(q) \textrm{ if } p <_P q \textrm{ and } p,q \textrm{ form a nonascending pair,}\\ & f(gp) = f(p) \textrm{ for all }p \in P,\ g \in G\}
    \end{split}\]
    and
    \[\begin{split}
    \mathscr{O}(P/G, \overline{\varepsilon}) = \{&\varphi\colon P/G \to [0,1] \textrm{ such that }\\ &\varphi(\mathcal{O}) \geq \varphi(\mathcal{O}') \textrm{ if } \mathcal{O} \leq_{P/G} \mathcal{O}',\\ &\varphi(\mathcal{O}) > \varphi(\mathcal{O}') \textrm{ if } \mathcal{O} <_{P/G} \mathcal{O}' \textrm{ and }\mathcal{O},\mathcal{O}' \textrm{ form a nonascending pair}\}. 
    \end{split}\]
    Note that the last condition in the definition of $\mathscr{O}(P,\varepsilon)^G$ is equivalent to the natural one, i.e., $gf = f$ for every $g \in G$ (which translates to $f(g^{-1}p) = f(p)$ for all $p \in P$, $g \in G$).

    Denote by $\pi \colon P \to P/G$ the natural projection and denote by $\lambda\colon P/G \to P$ a section of $\pi$, i.e., a map obtained by choosing a representative $\lambda(\mathcal{O})$ for each $G$-orbit $\mathcal{O}$. Note that, by \Cref{lem:sign-graded quotient}, one has that $\overline\varepsilon(\pi(p) \lessdot \pi(p')) = \varepsilon(p \lessdot p')$ for all $p \lessdot p' \in E(P)$ and that $\varepsilon(\lambda(\mathcal{O}) \lessdot \lambda(\mathcal{O}')) = \overline\varepsilon(\mathcal{O} \lessdot \mathcal{O}')$ for all $\mathcal{O} \lessdot \mathcal{O}' \in E(P/G)$.

    We now want to define two maps
    \[
    \begin{tikzcd}
        {\mathscr{O}(P,\varepsilon)^G} \arrow[r, yshift=0.75ex, "\Pi"] & {\mathscr{O}(P/G,\overline{\varepsilon})} \arrow[l, yshift=-0.75ex, "\Lambda"]
    \end{tikzcd}
    \]
    realizing the unimodular equivalence between $\mathscr{O}(P,\varepsilon)^G$ and $\mathscr{O}(P/G,\overline{\varepsilon})$. We set $\Lambda(\varphi) := \varphi \circ \pi$ and $\Pi(f) := f \circ \lambda$ (in other words, $\Lambda$ is the result of precomposing by $\pi$ and $\Pi$ is the result of precomposing by $\lambda$).

    Let us first check that $\Lambda$ and $\Pi$ are well-defined. Given $\varphi \in \mathscr{O}(P/G, \overline{\varepsilon})$, we need to check that $\Lambda(\varphi) \in \mathscr{O}(P, \varepsilon)^G$. If $p \leq_P q$, then $Gp \leq_{P/G} Gq$ and hence $\varphi(Gp) \geq \varphi(Gq)$, i.e., $\Lambda(\varphi)(p) \geq \Lambda(\varphi)(q)$, as desired. If $p <_P q$ and $p,q$ form a nonascending pair, then $Gp <_{P/G} Gq$ and $Gp,Gq$ form a nonascending pair, and thus $\Lambda(\varphi)(p) > \Lambda(\varphi)(q)$. Finally, for all $g \in G$ and $p \in P$ one has that $\Lambda(\varphi)(gp) = \varphi(Gp) = \Lambda(\varphi)(p)$. Now pick $f \in \mathscr{O}(P,\varepsilon)^G$. Since $f$ takes the same value on each representative of any given $G$-orbit $\mathcal{O}$, one has that $\Pi(f)(\mathcal{O})$ is uniquely determined. To check that $\Pi(f) \in \mathscr{O}(P/G, \overline{\varepsilon})$ it is enough to recall that, if $\mathcal{O}_0 \lessdot \mathcal{O}_1 \lessdot \ldots \lessdot \mathcal{O}_m$, then there exist $p_i \in \mathcal{O}_i$ such that $p_i \lessdot p_{i+1}$ and $\varepsilon(p_i \lessdot p_{i+1}) = \overline{\varepsilon}(\mathcal{O}_i \lessdot \mathcal{O}_{i+1})$ for every $i \in \{0, 1, \ldots, m-1\}$.

    It is left to the reader to check that $\Lambda \circ \Pi = \mathrm{id}_{\mathscr{O}(P,\varepsilon)^G}$, that $\Pi \circ \Lambda = \mathrm{id}_{\mathscr{O}(P/G, \overline{\varepsilon})}$, and that $\Lambda$ and $\Pi$ yield unimodular transformations between the lattices $\bigoplus_{\mathcal{O} \in P/G}\mathbb{Z}\cdot\mathbf{e}^*_{\mathcal{O}}$ and $\bigoplus_{\mathcal{O} \in P/G}\mathbb{Z} \cdot (\sum_{p \in \mathcal{O}}\mathbf{e}^*_p)$.  
\end{proof}

Note that, if $P$ is $\varepsilon$-consistent and $m$ is a positive integer, then the $m$-th dilation of $\mathscr{O}(P,\varepsilon)$ consists of all $f \in P \to [0,m]$ such that inequalities \eqref{eq:defining inequalities of an order polytope} hold. Moreover, the elements of $m\mathscr{O}(P,\varepsilon)$ that belong to the lattice $\mathbb{Z}^P$ correspond bijectively to the $(P,\varepsilon)$-partitions\footnote{To be consistent with the literature one would need to talk about $(P,\omega)$-partitions, where the vertex labeling $\omega$ is obtained from $\varepsilon$ as in \Cref{rem:vertex labeling}.} with largest part at most $m+1$. In view of this, we define $0\mathscr{O}(P,\varepsilon)$ to be the singleton consisting of the zero function if $\varepsilon=\mathbf{1}$, and the empty set otherwise.

\subsection{Representation theory of finite groups}\label{sec:rep}

We briefly review some fundamental notions about the representation theory of finite groups. For a more detailed introduction, we refer the reader to, e.g.,~\cite{serre}.

Let $G$ be a finite group. 
A \textit{complex representation} of $G$ is a group homomorphism $\rho\colon G \rightarrow \GL(V)$ for some finite-dimensional vector space $V$ over $\CC$.
A complex representation can be identified with a $\CC G$-module, where $\CC G$ is the complex group algebra of $G$: on the one hand, given a complex representation $\rho\colon G \rightarrow \GL(V)$, we associate a $\CC G$-module structure with $V$ 
by setting $g \cdot v := \rho(g)v$ for $g \in G$ and $v \in V$; 
on the other hand, given a $\CC G$-module, we associate with it a group homomorphism $\rho$ defined by $\rho(g)(v):=g \cdot v$ for $v \in V$. 
We say that a complex representation $\rho\colon G \rightarrow \GL(V)$ is \textit{irreducible} 
if there is no $0 \neq W \subsetneq V$ such that $\rho(g) W =W$ for all $g \in G$. 
The \textit{character} of a complex representation $\rho\colon G \rightarrow \GL(V)$ is the map $\chi\colon G \rightarrow \CC$ 
which sends each $g \in G$ to the trace of the matrix $\rho(g)$. 
The \textit{trivial character}, denoted by $1$, is an irreducible character $1\colon G \rightarrow \CC$ defined by $1(g)=1$ for each $g \in G$. 

The (complex) \emph{ring of virtual characters} $\mathrm{R}_G$ of $G$ is the subring of the ring of complex class functions generated by all integer multiples of characters. More concretely, the ring $\mathrm{R}_G$ is a free abelian group whose basis consists of irreducible $G$-characters. If $\rho$ and $\theta$ are $G$-representations, then $\chi_{\rho} + \chi_{\theta} = \chi_{\rho \oplus \theta}$ and $\chi_{\rho} \cdot \chi_{\theta} = \chi_{\rho \otimes \theta}$; the addition and multiplication of $\mathrm{R}_G$ are induced by such relations. Note that the trivial character $1$ yields the multiplicative identity of $\mathrm{R}_G$. We say that an element of $\mathrm{R}_G$ is \emph{effective} if its coefficients with respect to the basis of irreducible $G$-characters are all nonnegative integers. This happens precisely when such an element corresponds to an actual $G$-character. Note that, since $G$ is finite and we work over the complex numbers, the ring $\mathrm{R}_G$ carries the same information as the Grothendieck ring of representations of $G$ over $\CC$.

We also recall induced representations and their characters. Let $H$ be a (not necessarily normal) subgroup of $G$ and fix a system of representatives $R$ of the set of left cosets $G/H$. 
Let $\theta\colon H \rightarrow \GL(W)$ be a representation of $H$, where $W$ is a vector space over $\CC$. 
We say that the representation $\rho\colon G \rightarrow \GL(V)$, where $V$ is a finite-dimensional vector space over $\CC$, of $G$ is \textit{induced} by $\theta$ 
if $V=\bigoplus_{\sigma \in G/H}W_\sigma$, where $W_\sigma=\{w \in W \mid g \cdot w = w \text{ for every } g \in \sigma\}$. (Said otherwise, if $W$ is a $\mathbb{C}H$-module, then the induced representation corresponds to the $\mathbb{C}G$-module obtained as $\CC G \otimes_{\CC H} W$.) We will write that $\rho = \mathrm{Ind}^G_H(\theta)$. 
Let $\chi_\rho$ and $\chi_\theta$ be the characters of $G$ and $H$ corresponding to $\rho$ and $\theta$, respectively. 
If $\rho$ is induced by $\theta$, then $\chi_\rho$ can be described in terms of $\theta$ as follows: 
\begin{align}\label{eq:induced}
\chi_\rho(u)=\sum_{\substack{r \in R \\ r^{-1}ur \in H}}\chi_\theta(r^{-1}ur). 
\end{align} 
See, e.g., \cite[Theorem 12]{serre}. Note that, if $\theta = 1$, then the induced representation $\mathrm{Ind}^G_H(1)$ equals the permutation representation of $G$ on the set of left cosets $G/H$. In particular, if $\theta=1$ and $H = \{e\}$, one has that $\mathrm{Ind}^G_{\{e\}}(1)$ equals the regular representation of $G$.

Finally, if $H$ is a subgroup of $G$ and $\rho\colon G \to \GL(V)$ is a representation of $G$, one can consider the representation $\theta$ of $H$ obtained by restricting $\rho$ to $H$. We will say that $\theta$ is a \emph{restriction} of $\rho$, and write $\theta = \mathrm{Res}^G_H(\rho)$. Induced and restricted representations are linked by Frobenius reciprocity: see, e.g., \cite[Theorem 13]{serre}. 

\subsection{Equivariant Ehrhart theory for order polytopes} \label{subsec:equivariant Ehrhart for order polytopes}

Let $P$ be an $\varepsilon$-consistent poset. 
The group $\mathfrak{S}_{P}$ of bijections of the set $P$ acts linearly on the lattice $M := \mathbb{Z}^P$ by \[\begin{split}\xi\colon\mathfrak{S}_P &\to \mathrm{GL}(\mathbb{Z}^P)\\g &\mapsto \xi_g \colon f \mapsto f \circ g^{-1}\end{split}\]
It is useful to extend such an action to $\widetilde{M} := \mathbb{Z}^P \oplus \mathbb{Z}$ by mapping $g$ to $\widetilde{\xi}_g \colon f \oplus m \to \xi_g(f) \oplus m$; one then gets complex representations $\xi_{\mathbb{C}}$ and $\widetilde\xi_{\mathbb{C}}$ (respectively of $M_{\mathbb{C}} := \mathbb{C}^P$ and $\widetilde M_{\mathbb{C}} := \mathbb{C}^P \oplus \mathbb{C}$) by tensoring $\xi$ and $\widetilde\xi$ with $\mathbb{C}$. 

Now let $G$ be a subgroup of $\mathrm{Aut}(P,\varepsilon)$, which in turn sits naturally as a subgroup of $\mathfrak{S}_{P}$. Let $\eta_G := \mathrm{Res}^{\mathfrak{S}_P}_{G}\xi_{\mathbb{C}}$ (respectively, $\widetilde\eta_G := \mathrm{Res}^{\mathfrak{S}_P}_{G}\widetilde\xi_{\mathbb{C}}$) denote the restriction of $\xi_{\mathbb{C}}$ (respectively, $\widetilde\xi_{\mathbb{C}}$) to $G$. Then, for each nonnegative integer $m$, one has that $\widetilde\eta_G$ preserves $m\mathscr{O}(P,\varepsilon) \oplus \{m\}$ and gives rise to a permutation action on its lattice points. Let $\chi_{m\mathscr{O}(P,\varepsilon)}$ be the associated complex permutation representation. Following Stapledon \cite{stapledon11}, we can then write the equivariant Ehrhart series of $\mathscr{O}(P,\varepsilon)$
    
    \[\EE(\mathscr{O}(P,\varepsilon),\eta_G;t) := \sum_{m \in \mathbb{Z}_{\geq 0}}\chi_{m\mathscr{O}(P,\varepsilon)}t^m \in \mathrm{R}_G[[t]]\]
and define
    \[\begin{split}{\mathfrak{h}^*}^G_{P,\varepsilon}(t) &:= \EE(\mathscr{O}(P,\varepsilon),\eta_G;t)\cdot\det(\mathrm{Id}-\widetilde\eta_Gt)\\ &\phantom{:}=\EE(\mathscr{O}(P,\varepsilon),\eta_G;t)\cdot(1-t)\cdot\det(\mathrm{Id}-\eta_Gt),\end{split}\]
where $\det(\mathrm{Id}-\eta_Gt)$ is the polynomial in $\mathrm{R}_G[t]$ defined as $\sum_{i=0}^{|P|}(-1)^i\chi_{\wedge^i\eta_G}t^i$ (which equals the inverse of $\sum_{j\geq 0}\chi_{\mathrm{Sym}^j\eta_G} \cdot t^j$ in $\mathrm{R}_G[[t]]$: see \cite[Lemma 3.1]{stapledon11}), and likewise for $\widetilde{\eta}_G$. 
Note that, unlike the non-equivariant case, ${\mathfrak{h}^*}^G_{P,\varepsilon}(t)$ is a formal series that might or might not be a polynomial.

If $g \in G$, we denote by $M_g$ the vector subspace of $M_{\mathbb{R}} = M \otimes_{\mathbb{Z}} \mathbb{R}$ that is fixed pointwise by $g$. If $g$ has cycle type $(\mu_1, \mu_2, \ldots, \mu_r)$, one checks that $M_g$ is $r$-dimensional and that $\det(\mathrm{Id}-\eta_Gt) = \prod_{j=1}^r(1-t^{\mu_j})$ \cite[Proof of Lemma 9.3]{stapledon11}. Moreover, one has that $\dim \mathscr{O}(P,\varepsilon)^g = \dim M_g = r$, since \[\mathrm{relint}(\mathscr{O}(P/\langle g \rangle)) \cong \mathrm{relint}(\mathscr{O}(P)^g) \subseteq \mathscr{O}(P,\varepsilon)^g \subseteq \mathscr{O}(P)^g \cong \mathscr{O}(P/\langle g \rangle),\] where we used the unimodular equivalence from \Cref{lem:order polytope of a quotient poset} and the fact that the order polytope $\mathscr{O}(P/\langle g \rangle)$ is full-dimensional inside the $r$-dimensional affine space $\bigoplus_{\mathcal{O} \in P / \langle g \rangle}\mathbb{R}\mathrm{e}^*_{\mathcal{O}}$.

\begin{Lemma} \label{lem:evaluation at g} 
    Let $P$ be an $\varepsilon$-consistent poset, let $G$ be a subgroup of $\mathrm{Aut}(P,\varepsilon)$ and let $g \in G$. Then \[{\mathfrak{h}^*}^{G}_{P,\varepsilon}(t)(g) = h^*(\mathscr{O}(P,\varepsilon)^{g}; t) \cdot \det(\mathrm{Id}-\eta_G(g)t)|_{M_g^{\perp}}.\]
    If $g$ has cycle type $(\mu_1, \ldots, \mu_r)$, it also holds that
    \[{\mathfrak{h}^*}^{G}_{P,\varepsilon}(t)(g) = h^*(\mathscr{O}(P,\varepsilon)^{g}; t) \cdot \prod_{j=1}^{r}(1+t+t^2+\ldots+t^{\mu_j-1}).\]
\end{Lemma}
\begin{proof}
Note first that one has that
 \[\EE(\mathscr{O}(P,\varepsilon),\eta_G;t)(g) = \sum_{m \in \mathbb{Z}_{\geq 0}}\chi_{m\mathscr{O}(P,\varepsilon)}(g)\cdot t^m = \mathrm{Ehr}(\mathscr{O}(P,\varepsilon)^g; t),\]
 since \[\chi_{m\mathscr{O}(P,\varepsilon)}(g) = |\{x \in m\mathscr{O}(P,\varepsilon)\cap M \mid gx = x\}| = |m\mathscr{O}(P,\varepsilon)^g \cap M|\]
 for every $m \in \mathbb{Z}_{\geq 0}$.
 Since $M_{\mathbb{R}} = M_g \oplus M_g^{\perp}$ and \[\begin{split}\det(\mathrm{Id}-\eta_G(g)t) &= (1-t)^{\dim(M_g)}\det(\mathrm{Id}-\eta_G(g)t)|_{M_g^{\perp}} \\&= (1-t)^{\dim(\mathscr{O}(P,\varepsilon)^g)}\det(\mathrm{Id}-\eta_G(g)t)|_{M_g^{\perp}},\end{split}\]
 the first claim follows. To prove the second claim it is sufficient to recall that, if $g$ has cycle type $(\mu_1, \ldots, \mu_r)$, then $\dim M_g = r$ and $\det(\mathrm{Id}-\eta_G(g)t) = \prod_{j=1}^r (1-t^{\mu_j})$.
\end{proof}

Let $P$ and $Q$ be finite posets and let $G$ and $H$ be subgroups of respectively $\mathrm{Aut}(P)$ and $\mathrm{Aut}(Q)$. Then $G \times H$ acts on the ordinal sum $P \oplus Q$ by \[(g, h) \mapsto \left(x \mapsto \begin{cases}gx & x \in P\\ hx & x \in Q\end{cases} \right) \] and on the lattice $\mathbb{Z}^P \oplus \mathbb{Z}^Q$ by $(g, h) \mapsto (f \oplus f' \mapsto f \circ g^{-1} \oplus f' \circ h^{-1})$.

\begin{Lemma} \label{lem:ordinal sum}
Let $P$ be an $\varepsilon$-consistent poset, let $Q$ be a $\delta$-consistent poset and let $G$ and $H$ be subgroups of respectively $\mathrm{Aut}(P, \varepsilon)$ and $\mathrm{Aut}(Q, \delta)$. Then:
\begin{enumerate}[label=(\roman*)]
\item $(P \oplus_{1} Q)/ G \times H$ is an $\overline\varepsilon \oplus_1 \overline\delta$-consistent poset that is isomorphic (as a sign-consistent poset) to $(P/G \oplus Q/H, \ \overline\varepsilon\oplus_1 \overline\delta)$;
\item it holds that 
\[{\mathfrak{h}^*}^{G \times H}_{P \oplus Q, \ \varepsilon \oplus_1 \delta}(t) = {\mathfrak{h}^*}^G_{P, \varepsilon}(t) \cdot {\mathfrak{h}^*}^H_{Q, \delta}(t),\]
where ${\mathfrak{h}^*}^G_{P,\varepsilon}(t)$ is considered as a $G \times H$-character by ${\mathfrak{h}^*}^G_{P,\varepsilon}(t)(g,h) := {\mathfrak{h}^*}^G_{P,\varepsilon}(t)(g)$ (and similarly for ${\mathfrak{h}^*}^H_{Q,\delta}(t)$).
\end{enumerate}
\end{Lemma}
\begin{proof}
Let us prove part (ii). Fix $(g,h) \in G \times H$ where $g$ has cycle type $(\mu_1, \mu_2, \ldots, \mu_r)$ and $h$ has cycle type $(\nu_1, \nu_2, \ldots, \nu_s)$. Then, by \Cref{lem:evaluation at g} and \Cref{lem:order polytope of a quotient poset}, one has that ${\mathfrak{h}^*}^{G \times H}_{P \oplus Q, \ \varepsilon \oplus_1 \delta}(g, h)$ equals 
\[\begin{split} &h^*(\mathscr{O}(P \oplus Q, \varepsilon \oplus_1 \delta)^{(g,h)};t) \cdot \prod_{j=1}^{r}\frac{1-t^{\mu_j}}{1-t} \cdot \prod_{k=1}^{s}\frac{1-t^{\nu_k}}{1-t}\\ =\ &h^*(\mathscr{O}(P \oplus Q/{\langle (g,h) \rangle}, \overline{\varepsilon} \oplus_1 \overline{\delta});t) \cdot \prod_{j=1}^{r}\frac{1-t^{\mu_j}}{1-t} \cdot \prod_{k=1}^{s}\frac{1-t^{\nu_k}}{1-t}.\end{split}\]
By part (i) and \cite[Proposition 3.3]{branden04}, this quantity equals
\[\begin{split} &h^*(\mathscr{O}((P / \langle g \rangle, \overline\varepsilon) \oplus_1 (Q/\langle h \rangle, \overline\delta));t) \cdot \prod_{j=1}^{r}\frac{1-t^{\mu_j}}{1-t} \cdot \prod_{k=1}^{s}\frac{1-t^{\nu_k}}{1-t}\\ = \ &h^*(\mathscr{O}(P/\langle g \rangle, \overline\varepsilon);t)\cdot h^*(\mathscr{O}(Q/\langle h \rangle, \overline\delta);t) \cdot  \prod_{j=1}^{r}\frac{1-t^{\mu_j}}{1-t} \cdot \prod_{k=1}^{s}\frac{1-t^{\nu_k}}{1-t}\\ = \ &\left(h^*(\mathscr{O}(P/\langle g \rangle, \overline\varepsilon);t)\cdot \prod_{j=1}^{r}\frac{1-t^{\mu_j}}{1-t}\right) \cdot \left(h^*(\mathscr{O}(Q/\langle h \rangle, \overline\delta);t) \cdot \prod_{k=1}^{s}\frac{1-t^{\nu_k}}{1-t}\right),\end{split}\]
which is the same as ${\mathfrak{h}^*}^G_{P, \varepsilon}(t)(g) \cdot {\mathfrak{h}^*}^H_{Q, \delta}(t)(h)$ by \Cref{lem:order polytope of a quotient poset} and \Cref{lem:evaluation at g}.
\end{proof}

\begin{Lemma} \label{lem:different sign-graded structures}
Let $P$ be $\varepsilon$-graded. Then $P$ is also $\varepsilon_{\mathrm{par}}$-graded. Moreover, in $\mathrm{R}_G[[t]]$ one has \begin{equation} \label{eq:equality of equivariant h-vectors}
{\mathfrak{h}^*}^G_{P,\varepsilon}(t) = t^{\frac{r_P(\varepsilon_{\mathrm{par}})-r_P(\varepsilon)}{2}} \cdot {\mathfrak{h}^*}^G_{P,\varepsilon_{\mathrm{par}}}(t). 
\end{equation}
\end{Lemma}
\begin{proof}
    The statement about $P$ being parity-graded follows from \Cref{rem:parity-graded}. To prove that \eqref{eq:equality of equivariant h-vectors} holds, let us evaluate both sides of the expression at an arbitrary $g \in G$ (with cycle type $(\mu_1, \ldots, \mu_r)$). 
    
    By \Cref{lem:evaluation at g} and \Cref{lem:order polytope of a quotient poset}, one has that \[\begin{split}{\mathfrak{h}^*}^G_{P,\varepsilon}(t)(g) &= h^*(\mathscr{O}(P,\varepsilon)^{g};t) \cdot \prod_{j=1}^{r}(1+t+t^2+\ldots+t^{\mu_j-1})\\
    &= h^*(\mathscr{O}(P/\langle g \rangle,\overline{\varepsilon});t) \cdot \prod_{j=1}^{r}(1+t+t^2+\ldots+t^{\mu_j-1}).\end{split}\]

    Since $\mathrm{Aut}(P,\varepsilon_{\mathrm{par}}) = \mathrm{Aut}(P)$ by \Cref{rem:parity-graded}, one has that $g$ is an element of both $\mathrm{Aut}(P,\varepsilon)$ and $\mathrm{Aut}(P,\varepsilon_{\mathrm{par}})$; \Cref{lem:sign-graded quotient} then implies that $P/\langle g \rangle$ is both $\varepsilon$-graded and $\varepsilon_{\mathrm{par}}$-graded, and the respective rank functions are preserved when passing to the quotient. It then follows from \cite[Corollary 2.4]{branden04} that \[h^*(\mathscr{O}(P/\langle g \rangle,\overline{\varepsilon});t) = t^{\frac{r_P(\varepsilon_{\mathrm{par}})-r_P(\varepsilon)}{2}} \cdot h^*(\mathscr{O}(P/\langle g \rangle,\overline{\varepsilon_{\mathrm{par}}});t),\]
    and thus, applying \Cref{lem:evaluation at g} again,
    \[\begin{split}
    {\mathfrak{h}^*}^G_{P,\varepsilon}(t)(g) &= t^{\frac{r_P(\varepsilon_{\mathrm{par}})-r_P(\varepsilon)}{2}} \cdot h^*(\mathscr{O}(P/\langle g \rangle,\overline{\varepsilon_{\mathrm{par}}});t) \cdot \prod_{j=1}^{r}(1+t+t^2+\ldots+t^{\mu_j-1})\\
    &= t^{\frac{r_P(\varepsilon_{\mathrm{par}})-r_P(\varepsilon)}{2}} \cdot {\mathfrak{h}^*}^G_{P,\varepsilon_{\mathrm{par}}}(t)(g),
    \end{split}\]
    whence the claim follows.
\end{proof}

We are now ready to state the first main result of the present paper, i.e., a way to access the numerator of the equivariant Ehrhart series of $\mathscr{O}(P,\varepsilon)$ via its saturations.

\begin{Theorem} \label{thm:main theorem on h^*}
Let $P$ be an $\varepsilon$-consistent poset and let $G$ be a subgroup of $\mathrm{Aut}(P, \varepsilon)$. Let $\mathcal{Q}$ be the set of all saturations of $(P,\varepsilon)$ and let $(Q_1,\delta_1),\ldots, (Q_k,\delta_k)$ be representatives of the $G$-orbits of $\mathcal{Q}$. Then
\begin{equation} \label{eq:main theorem}
{\mathfrak{h}^*}^G_{P,\varepsilon}(t) = \sum_{i=1}^k \mathrm{Ind}^G_{\mathrm{stab}_G(Q_i,\delta_i)}\mathrm{Res}^{\mathrm{Aut}(Q_i, \delta_i)}_{\mathrm{stab}_G(Q_i,\delta_i)}{\mathfrak{h}^*}^{\mathrm{Aut}(Q_i, \delta_i)}_{Q_i, \delta_i}(t).
\end{equation}
\end{Theorem}

\begin{proof}
    We prove the statement by checking that the two sides of \eqref{eq:main theorem} yield the same quantity when evaluated at an arbitrary $g \in G$. Let $(\mu_1, \ldots, \mu_r)$ be the cycle type of $g$.

    Applying \Cref{lem:evaluation at g} to the left hand side of \eqref{eq:main theorem}, one has that
    \begin{equation*}
        {\mathfrak{h}^*}^G_{P,\varepsilon}(t)(g) = h^*(\mathscr{O}(P,\varepsilon)^{g};t)\cdot \prod_{j=1}^r \frac{1-t^{\mu_j}}{1-t}.
    \end{equation*}

    \Cref{lem:order polytope of a quotient poset} yields that $h^*(\mathscr{O}(P,\varepsilon)^{g};t) = h^*(\mathscr{O}(P/\langle g \rangle, \overline{\varepsilon});t)$; then, by \cite[Theorem 3.2]{branden04} and by \Cref{lem:saturations of a quotient}, it follows that
    \begin{equation*}
    h^*(\mathscr{O}(P/\langle g \rangle, \overline{\varepsilon});t) = \sum_{\substack{(Q',\delta') \text{ saturation}\\ \text{of } (P/\langle g \rangle, \overline{\varepsilon})}}h^*(\mathscr{O}(Q',\delta');t) = \sum_{\substack{(Q,\delta) \in \mathcal{Q}\\ \text{s.t.~}\langle g \rangle \subseteq \mathrm{Aut}(Q,\delta)}}h^*(\mathscr{O}(Q/\langle g \rangle,\overline{\delta});t).
    \end{equation*}

    Now note that $\langle g \rangle \subseteq \mathrm{Aut}(Q,\delta)$ if and only if $g \in \mathrm{Aut}(Q,\delta) \cap G$, which in turn is equivalent to asking that $g$ lies in the stabilizer of $(Q,\delta)$ with respect to the $G$-action on the set $\mathcal{Q}$ of saturations of $(P,\varepsilon)$ by \Cref{lem:stabilizer}. Hence, another application of \Cref{lem:order polytope of a quotient poset} yields that

    \begin{equation*}
        \sum_{\substack{(Q,\delta) \in \mathcal{Q}\\ \text{s.t.~}\langle g \rangle \subseteq \mathrm{Aut}(Q,\delta)}}h^*(\mathscr{O}(Q/\langle g \rangle,\overline{\delta});t) = \sum_{\substack{(Q,\delta) \in \mathcal{Q}\\ \text{s.t.~}g \in \mathrm{stab}_G(Q,\delta)}}h^*(\mathscr{O}(Q,\delta)^{g};t).
    \end{equation*}

    Taking into account that the set $\mathcal{Q}$ of saturations of $(P,\varepsilon)$ is partitioned into $G$-orbits $G(Q_1,\delta_1), \ldots, G(Q_k,\delta_k)$ we get that

    \begin{equation} \label{eq:LHS through saturations}
    \sum_{\substack{(Q,\delta) \in \mathcal{Q}\\ \text{s.t.~}g \in \mathrm{stab}_G(Q,\delta)}}h^*(\mathscr{O}(Q,\delta)^g;t) = \sum_{i=1}^{k}\sum_{\substack{g'(Q_i,\delta_i) \in G(Q_i,\delta_i)\\g \in \mathrm{stab}_G (g'(Q_i,\delta_i))}}h^*(\mathscr{O}(g'(Q_i,\delta_i))^{g};t),
    \end{equation}
    whence, recalling the orbit-stabilizer theorem and the fact that $g \in \mathrm{stab}_G(g'(Q_i,\delta_i))$ if and only if $(g')^{-1}gg' \in \mathrm{stab}_G(Q_i,\delta_i)$ (since both statements are equivalent to asking that $gg'(Q_i,\delta_i) = g'(Q_i,\delta_i)$), we obtain that the quantities in \eqref{eq:LHS through saturations} equal

    \begin{equation} \label{eq:last manipulation of LHS}
    \sum_{i=1}^{k}\sum_{\substack{g'\mathrm{stab}_G(Q_i,\delta_i) \in G/\mathrm{stab}_G(Q_i,\delta_i)\\(g')^{-1}gg' \in \mathrm{stab}_G (Q_i,\delta_i)}}h^*(\mathscr{O}(g'(Q_i,\delta_i))^{g};t).
    \end{equation}

    Now note that $\mathscr{O}(g'(Q_i,\delta_i))^{g} = g'(\mathscr{O}(Q_i,\delta_i)^{g'^{-1}gg'})$ and thus, since multiplication by $g'$ yields a unimodular transformation, the $h^*$-polynomials of $\mathscr{O}(g'(Q_i,\delta_i))^{g}$ and $\mathscr{O}(Q_i,\delta_i)^{g'^{-1}gg'}$ coincide. Thus, \eqref{eq:last manipulation of LHS} yields that

    \begin{equation} \label{eq:reformulation of LHS}
    {\mathfrak{h}^*}^G_{P,\varepsilon}(t)(g) = \left(\prod_{j=1}^r \frac{1-t^{\mu_j}}{1-t}\right) \cdot \sum_{i=1}^{k}\sum_{\substack{g'\mathrm{stab}_G(Q_i,\delta_i) \in G/\mathrm{stab}_G(Q_i,\delta_i)\\g'^{-1}gg' \in \mathrm{stab}_G (Q_i,\delta_i)}}h^*(\mathscr{O}(Q_i,\delta_i)^{g'^{-1}gg'};t).
    \end{equation}

    Let us now evaluate the right hand side of \eqref{eq:main theorem} at $g$. One gets that

    \[\begin{split}
        &\sum_{i=1}^k \mathrm{Ind}^G_{\mathrm{stab}_G(Q_i,\delta_i)}\mathrm{Res}^{\mathrm{Aut}(Q_i, \delta_i)}_{\mathrm{stab}_G(Q_i,\delta_i)}{\mathfrak{h}^*}^{\mathrm{Aut}(Q_i, \delta_i)}_{Q_i, \delta_i}(t)(g)\\
        =&\sum_{i=1}^k \sum_{\substack{g'\mathrm{stab}_G(Q_i,\delta_i) \in G/\mathrm{stab}_G(Q_i,\delta_i)\\g'^{-1}gg' \in \mathrm{stab}_G(Q_i,\delta_i)}}\mathrm{Res}^{\mathrm{Aut}(Q_i, \delta_i)}_{\mathrm{stab}_G(Q_i,\delta_i)}{\mathfrak{h}^*}^{\mathrm{Aut}(Q_i, \delta_i)}_{Q_i, \delta_i}(t)(g'^{-1}gg')\\
        =&\sum_{i=1}^k \sum_{\substack{g'\mathrm{stab}_G(Q_i,\delta_i) \in G/\mathrm{stab}_G(Q_i,\delta_i)\\g'^{-1}gg' \in \mathrm{stab}_G(Q_i,\delta_i)}}{\mathfrak{h}^*}^{\mathrm{Aut}(Q_i, \delta_i)}_{Q_i, \delta_i}(t)(g'^{-1}gg')\\
        =&\sum_{i=1}^k \sum_{\substack{g'\mathrm{stab}_G(Q_i,\delta_i) \in G/\mathrm{stab}_G(Q_i,\delta_i)\\g'^{-1}gg' \in \mathrm{stab}_G(Q_i,\delta_i)}}h^*(\mathscr{O}(Q_i, \delta_i)^{g'^{-1}gg'};t) \cdot \det(\mathrm{Id}-\widetilde{\eta}_{\mathrm{Aut}(Q_i,\delta_i)}(g'^{-1}gg')t)_{(M_{g'^{-1}gg'})^{\perp}}\\
        =&\sum_{i=1}^k \sum_{\substack{g'\mathrm{stab}_G(Q_i,\delta_i) \in G/\mathrm{stab}_G(Q_i,\delta_i)\\g'^{-1}gg' \in \mathrm{stab}_G(Q_i,\delta_i)}}h^*(\mathscr{O}(Q_i, \delta_i)^{g'^{-1}gg'};t) \cdot \prod_{j=1}^{r}\frac{1-t^{\mu_j}}{1-t},
    \end{split}\]
    where the last equality comes from the fact that the cycle type of a permutation is invariant under conjugation. Comparing with \eqref{eq:reformulation of LHS} ends the proof.
\end{proof}

\section{Equivariant \texorpdfstring{$\gamma$}{gamma}-polynomials of graded posets} \label{sec:equivariant gamma-polynomials}

This section is devoted to the study of equivariant $\gamma$-polynomials. To begin with, let us prove that these are well-defined when we work with order polytopes of ($\mathbf{1}$-)graded posets: 

\begin{Lemma} \label{lem:equivariant h*_P is a palindromic polynomial}
Let $P$ be a $\mathbf{1}$-graded poset and let $G$ be a subgroup of $\mathrm{Aut}(P)$. Then ${\mathfrak{h}^*}^G_P(t)$ is a palindromic polynomial in $\mathrm{R}_G[t]$ with effective coefficients, degree $|P|-r_P(\mathbf{1})-1$ and center of symmetry $\frac{|P|-r_P(\mathbf{1})-1}{2}$.
\end{Lemma}
\begin{proof}
D'Al\`i and Delucchi observed in \cite[Lemma 6.4]{dali-delucchi} that the order polytope $\mathscr{O}(P)$ admits a $G$-invariant triangulation. It then follows from \cite[Theorem 1.4]{stapledon24} that ${\mathfrak{h}^*}^G_P(t)$ is a polynomial, and moreover its coefficients are effective. Since Hibi \cite{hibi-distributive} proved that the (non-equivariant) $h^*$-polynomial of the order polytope of a $\mathbf{1}$-graded poset is palindromic, it follows from \cite[Corollary 6.9]{stapledon11} that so is its equivariant counterpart ${\mathfrak{h}^*}^G_P(t)$. The non-equivariant $h^*$-polynomial of an order polytope is known to have degree $|P|-r_P(\mathbf{1}) - 1$ (see, e.g., \cite[Corollary 2.4]{branden04}), and hence so does ${\mathfrak{h}^*}^G_P(t)$ by \cite[Corollary 6.6]{stapledon11}.
\end{proof}
    
\begin{Definition}
Let $P$ be a $\mathbf{1}$-graded poset and let $G$ be a subgroup of $\mathrm{Aut}(P)$. We call the \emph{equivariant $\gamma$-polynomial} of $P$ the polynomial $\boldsymbol{\gamma}^G_{P}(t) := \sum_{i=0}^{\lfloor s/2\rfloor}\boldsymbol{\gamma}_it^i \in \mathrm{R}_G[t]$, where $s = |P| - r_P(\mathbf{1}) - 1$ and the coefficients $\boldsymbol\gamma_i \in \mathrm{R}_G$ arise from \[{\mathfrak{h}^*}^G_{P}(t) = \sum_{i=0}^{\lfloor s/2\rfloor}\boldsymbol\gamma_it^i(1+t)^{s-2i}.\]
We say that $P$ is \emph{$\gamma$-effective} when $\boldsymbol\gamma_i$ is effective for every $i \in \{0, 1, \ldots, \lfloor \frac{s}{2}\rfloor\}$.
\end{Definition}

\begin{Lemma} \label{lem:gamma of order polytope}
    Let $P$ and $Q$ be $\mathbf{1}$-graded posets and let $G$ and $H$ be subgroups of $\mathrm{Aut}(P)$ and $\mathrm{Aut}(Q)$, respectively. Then:
    \[\boldsymbol{\gamma}_{P \oplus_1 Q}^{G \times H}(t) = \boldsymbol{\gamma}_P^G(t) \cdot \boldsymbol{\gamma}_Q^H(t),\]
    where ${\boldsymbol{\gamma}}^G_{P}(t)$ is considered as a $G \times H$-character by ${\boldsymbol{\gamma}}^G_{P}(t)(g,h) := {\boldsymbol{\gamma}}^G_{P}(t)(g)$ (and similarly for ${\boldsymbol{\gamma}}^H_{Q}(t)$).
\end{Lemma}
\begin{proof} 
This is a consequence of \Cref{lem:equivariant h*_P is a palindromic polynomial} and \Cref{lem:ordinal sum}(ii).
\end{proof}

There is a crucial case where the $\gamma$-effectiveness of $\mathscr{O}(P)$ is known and the $\gamma$-coefficients can be described explicitly; namely, when $P$ is an antichain, and $\mathscr{O}(P)$ is thus just the unit cube acted upon by the symmetric group $\Sf_P$. This is a result due to Shareshian and Wachs \cite[Corollary 3.2]{SW20} and, in more detail, to Horiguchi et al. \cite[Theorem 1.1]{HMSSS}. For the reader's convenience, we summarize the latter work in the next subsection.

\subsection{Equivariant Ehrhart series of cubes and their \texorpdfstring{$\gamma$}{gamma}-effectiveness}\label{subsec:cube}

In this subsection we recall some known facts on the (equivariant) $h^*$-polynomials of unit cubes. 

Let $C_d =[0,1]^d$ be the $d$-dimensional unit cube (which can be thought of as the order polytope of an antichain on $d$ elements). Then the $h^*$-polynomial of $C_d$ is well-known to equal
\[A_d(t):=(1-t)^{d+1} \sum_{m \geq 0}(m+1)^d t^m.\]
The polynomial $A_d(t)$ is called the $d$-th \textit{Eulerian polynomial}. 

For any $n \in \mathbb{Z}_{>0}$, we denote by $[n]$ the set $\{1,2,3,\ldots,n\}$. Given $w=w_1 \cdots w_d \in \Sf_d$ written in one-line notation, a \textit{descent} of $w$ is $i \in [d-1]$ with $w_i>w_{i+1}$, and $\des(w)$ denotes the number of descents of $w$. 
The following explicit description of Eulerian polynomials in terms of descents is well known: $\displaystyle A_d(t)=\sum_{w \in \Sf_d}t^{\des(w)}$. 
\begin{Remark}[{\cite{FS}}]\label{rem:cube}
The $\gamma$-polynomial of any Eulerian polynomial is also known. In fact, we can express $A_d(t)$ as follows: 
\[A_d(t)=\sum_{w \in \widehat{\Sf_d}}t^{\des(w)}(1+t)^{d-1-2\des(w)},\]
where we say that $w$ has a \textit{double descent} if there is $i \in [d-2]$ such that $i$ and $i+1$ are both descents, 
and has a \textit{final descent} if $d-1$ is a descent, and $\widehat{\Sf_d}=\{w \in \Sf_d \mid w \text{ has neither a double descent nor a final descent}\}$. 
\end{Remark}

Now, let us consider the equivariant Ehrhart series of $C_d$ with respect to the action of $\Sf_d$ permuting the coordinates. 
It is known by, e.g., \cite[Corollary 3.2]{SW20} that $\mathfrak{h}^*_{C_d}(t)$ is a palindromic polynomial that is $\gamma$-effective. In fact, we know that 
\begin{align}\label{eq:gamma}
\mathfrak{h}^*_{C_d}(t)=\sum_{i=0}^{\left\lfloor \frac{d-1}{2} \right\rfloor}\chi_{d,i}t^i(1+t)^{d-1-2i},
\end{align}
where the $\chi_{d,i}$'s are characters of actual representations of $\Sf_d$. We write the associated equivariant $\gamma$-polynomial as
\[\boldsymbol{\gamma}_{C_d}(t):=\sum_{i=0}^{\left\lfloor \frac{d-1}{2} \right\rfloor} \chi_{d,i}t^i.\]

In \cite{HMSSS}, an explicit way how to describe each $\chi_{d,i}$ is given; to state such a description, we need to introduce some terminology. Let $T$ be a standard Young tableau with $d$ boxes. 
\begin{itemize}
\item An index $i \in [d-1]$ is a \textit{descent} of $T$ if $i$ appears in a higher row than does $i+1$. 
(The notions of double descent and final descent are defined analogously.) 
\item Let $\des(T) := |\{ i \in [d-1] \mid i \text{ is a descent of }T\}|$. 
\item Let $\widehat{\SYT_d}$ denote the set of all standard Young tableaux with $d$ boxes having neither a double descent nor a final descent. 
\item Let $\lambda(T)$ be the Young diagram underlying $T$. 
\item Let $\chi^\lambda$ be the irreducible character of the Specht module corresponding to the conjugacy class $\lambda$ of $\Sf_d$. 
\end{itemize}
It is proved in \cite[Theorem 1.1]{HMSSS} that 
\begin{align}\label{eq:HMSSS}
\chi_{d,i} = \sum_{\substack{T \in \widehat{\SYT_d} \\ \des(T)=i}}\chi^{\lambda(T)} \;\;\text{for }i=0,1,\ldots,\left\lfloor \frac{d-1}{2} \right\rfloor.
\end{align}
For example, one has the following: 
\begin{equation}\label{eq:ex}
\begin{split}
&\chi_{3,0}=\chi^{\tiny{\ydiagram{3}}}, \;\; \chi_{3,1}=\chi^{\tiny{\ydiagram{2,1}}}; \quad 
\chi_{4,0}=\chi^{\tiny{\ydiagram{4}}}, \;\; \chi_{4,1}=\chi^{\tiny{\ydiagram{2,2}}}+2\chi^{\tiny{\ydiagram{3,1}}}; \\
&\chi_{5,0}=\chi^{\tiny{\ydiagram{5}}}, \;\; \chi_{5,1}=2\chi^{\tiny{\ydiagram{3,2}}}+3\chi^{\tiny{\ydiagram{4,1}}}, \;\; 
\chi_{5,2}=\chi^{\tiny{\ydiagram{3,2}}}+\chi^{\tiny{\ydiagram{3,1,1}}}+\chi^{\tiny{\ydiagram{2,2,1}}}. 
\end{split}
\end{equation}

We collect here the character tables of $\Sf_d$ for small $d$'s for later calculations.

\begin{center}
\begin{table}
\begin{tabular}{l|rrr}
$\Sf_3$ &$1$ &$\chi^{\tiny{\ydiagram{1,1,1}}}$ &$\chi^{\tiny{\ydiagram{2,1}}}$ \\ \hline
$e$ &$1$ &$1$ &$2$ \\
$(1 \, 2)$ &$1$ &$-1$ &$0$ \\
$(1 \, 2 \, 3)$ &$1$ &$1$ &$-1$
\end{tabular}
\quad\quad
\begin{tabular}{l|rrrrr}
$\Sf_4$ &$1$ &$\chi^{\tiny{\ydiagram{1,1,1,1}}}$ &$\chi^{\tiny{\ydiagram{2,2}}}$ &$\chi^{\tiny{\ydiagram{3,1}}}$ &$\chi^{\tiny{\ydiagram{2,1,1}}}$\\ \hline
$e$ &$1$ &$1$ &$2$ &$3$ &$3$ \\
$(1 \, 2)$ &$1$ &$-1$ &$0$ &$1$ &$-1$ \\
$(1 \, 2 \, 3)$ &$1$ &$1$ &$-1$ &$0$ &$0$ \\
$(1 \, 2 \, 3 \, 4)$ &$1$ &$-1$ &$0$ &$-1$ &$1$ \\
$(1 \, 2)(3 \, 4)$ &$1$ &$1$ &$2$ &$-1$ &$-1$
\end{tabular} 
\bigskip
\caption{The character tables of $\Sf_3$ and $\Sf_4$.}\label{tab:S_3}
\end{table}
\end{center}

\subsection{\texorpdfstring{$\gamma$}{Gamma}-effectiveness for arbitrary order polytopes}
We are now ready to state the second main result of the present paper, i.e., an extension of Br\"and\'en's proof of $\gamma$-positivity for graded posets to the equivariant setting.

\begin{Theorem}\label{thm:gamma-effective}
Let $P$ be a $\mathbf{1}$-graded poset and let $G$ be a subgroup of $\mathrm{Aut}(P)$.
Then the order polytope $\mathscr{O}(P)$ is $\gamma$-effective.
\end{Theorem}
\begin{proof}
\Cref{lem:equivariant h*_P is a palindromic polynomial} yields that ${\mathfrak{h}^*}^G_P(t)$ is a palindromic polynomial with center of symmetry $\frac{|P|-r_P(\mathbf{1})-1}{2}$; thus, the equivariant $\gamma$-polynomial $\boldsymbol\gamma^G_P(t)$ is well-defined.

Since $P$ is sign-graded, it can also be endowed with a parity-graded structure; let $\varepsilon_{\mathrm{par}}$ be the associated edge labeling. By \Cref{rem:parity-graded}, one has that every element of $G$ preserves $\varepsilon_{\mathrm{par}}$; moreover, by \Cref{lem:different sign-graded structures}, one has that
\begin{equation} \label{eq:equality of h-star in gamma-proof}
{\mathfrak{h}^*}^G_{P,\varepsilon_{\mathrm{par}}}(t) = t^{\frac{r_P(\mathbf{1}) - r_P(\varepsilon_{\mathrm{par}})}{2}} \cdot {\mathfrak{h}^*}^G_{P}(t).
\end{equation}
Let $\mathcal{Q}$ be the set of saturations of $(P,\varepsilon_{\mathrm{par}})$; by \Cref{rem:nice saturations}, any such saturation is naturally endowed with a parity-graded structure, but its underlying poset can also be given a $\mathbf{1}$-graded structure. Let $(Q_1, \varepsilon_{\mathrm{par}}), \ldots, (Q_k, \varepsilon_{\mathrm{par}})$ be representatives of the $G$-orbits of $\mathcal{Q}$. Again by \Cref{rem:nice saturations}, one has that $\mathrm{Aut}(Q_i,\varepsilon_{\mathrm{par}}) = \mathrm{Aut}(Q_i,\mathbf{1}) = \mathrm{Aut}(Q_i)$ for every $i$. By \Cref{thm:main theorem on h^*} and \Cref{lem:different sign-graded structures} one has that

\[\begin{split}{\mathfrak{h}^*}^G_{P,\varepsilon_{\mathrm{par}}}(t) &= \sum_{i=1}^k \mathrm{Ind}^G_{\mathrm{stab}_G(Q_i, \varepsilon_{\mathrm{par}})}\mathrm{Res}^{\mathrm{Aut}(Q_i)}_{\mathrm{stab}_G(Q_i, \varepsilon_\mathrm{par})}{\mathfrak{h}^*}^{\mathrm{Aut}(Q_i)}_{Q_i, \varepsilon_{\mathrm{par}}}(t)\\
&= \sum_{i=1}^k \mathrm{Ind}^G_{\mathrm{stab}_G(Q_i, \varepsilon_\mathrm{par})}\mathrm{Res}^{\mathrm{Aut}(Q_i)}_{\mathrm{stab}_G(Q_i, \varepsilon_{\mathrm{par}})}\left(t^{\frac{r_{Q_i}(\mathbf{1})-r_{Q_i}(\varepsilon_{\mathrm{par}})}{2}} \cdot {\mathfrak{h}^*}^{\mathrm{Aut}(Q_i)}_{Q_i, \mathbf{1}}(t)\right).\end{split}\]

Since $(Q_i,\varepsilon_{\mathrm{par}})$ is a saturation of $(P,\varepsilon_{\mathrm{par}})$, one has that $r_{Q_i}(\varepsilon_{\mathrm{par}}) = r_P(\varepsilon_\mathrm{par})$. Comparing with \eqref{eq:equality of h-star in gamma-proof} then yields that

\[t^{\frac{r_P(\mathbf{1}) - r_P(\varepsilon_{\mathrm{par}})}{2}} \cdot {\mathfrak{h}^*}^G_{P}(t) = \sum_{i=1}^k \mathrm{Ind}^G_{\mathrm{stab}_G(Q_i, \varepsilon_{\mathrm{par}})}\mathrm{Res}^{\mathrm{Aut}(Q_i)}_{\mathrm{stab}_G(Q_i, \varepsilon_{\mathrm{par}})}\left(t^{\frac{r_{Q_i}(\mathbf{1})-r_{P}(\varepsilon_{\mathrm{par}})}{2}} \cdot {\mathfrak{h}^*}^{\mathrm{Aut}(Q_i)}_{Q_i, \mathbf{1}}(t)\right)\]
and thus

\[{\mathfrak{h}^*}^G_{P}(t) = \sum_{i=1}^k \mathrm{Ind}^G_{\mathrm{stab}_G(Q_i, \varepsilon_{\mathrm{par}})}\mathrm{Res}^{\mathrm{Aut}(Q_i)}_{\mathrm{stab}_G(Q_i, \varepsilon_{\mathrm{par}})}\left(t^{\frac{r_{Q_i}(\mathbf{1})-r_{P}(\mathbf{1})}{2}} \cdot {\mathfrak{h}^*}^{\mathrm{Aut}(Q_i)}_{Q_i, \mathbf{1}}(t)\right)\]

It is our next goal to find out a way to describe $\boldsymbol{\gamma}^G_P(t)$. By \Cref{lem:equivariant h*_P is a palindromic polynomial}, one has that ${\mathfrak{h}^*}^{\mathrm{Aut}(Q_i)}_{Q_i, \mathbf{1}}(t)$ is a palindromic polynomial whose center of symmetry is $\frac{|Q|-r_Q(\mathbf{1})-1}{2}$. This implies that $t^{\frac{r_{Q_i}(\mathbf{1})-r_{P}(\mathbf{1})}{2}} \cdot {\mathfrak{h}^*}^{\mathrm{Aut}(Q_i)}_{Q_i, \mathbf{1}}(t)$ is a ``palindromic up to shift'' polynomial with the same center of symmetry as ${\mathfrak{h}^*}^G_{P,\mathbf{1}}(t)$, since \[\frac{r_{Q_i}(\mathbf{1})-r_P(\mathbf{1})}{2} + \frac{|Q_i| - r_{Q_i}(\mathbf{1})-1}{2} = \frac{|P| - r_P(\mathbf{1}) - 1}{2}\] (where we used the fact that $|Q_i| = |P|$, since $(Q_i, \varepsilon_{\mathrm{par}})$ is a saturation of $(P,\varepsilon_\mathrm{par})$). As a consequence, one has that

\begin{equation} \label{eq:gamma via saturations}
{\boldsymbol\gamma}^G_{P}(t) = \sum_{i=1}^k \mathrm{Ind}^G_{\mathrm{stab}_G(Q_i,\varepsilon_{\mathrm{par}})}\mathrm{Res}^{\mathrm{Aut}(Q_i)}_{\mathrm{stab}_G(Q_i,\varepsilon_{\mathrm{par}})}\left(t^{\frac{r_{Q_i}(\mathbf{1})-r_P(\mathbf{1})}{2}} \cdot {\boldsymbol\gamma}^{\mathrm{Aut}(Q_i)}_{Q_i}(t)\right).
\end{equation}

In particular, to prove the claim it is enough to show that the coefficients of each ${\boldsymbol\gamma}^{\mathrm{Aut}(Q_i)}_{Q_i}(t)$ are effective. By \Cref{rem:nice saturations}, the poset underlying $Q_i$ is the ordinal sum $A_{i0} \oplus A_{i1} \oplus \ldots \oplus A_{ik_i}$, where each $A_{ij}$ is an antichain; applying \Cref{lem:gamma of order polytope} then yields that
\[{\boldsymbol\gamma}^{\mathrm{Aut}(Q_i)}_{Q_i}(t) = {\boldsymbol\gamma}_{A_{i0}}^{\Sf_{A_{i0}}}(t) \cdot {\boldsymbol\gamma}_{A_{i1}}^{\Sf_{A_{i1}}}(t)\cdots {\boldsymbol\gamma}_{A_{ik_i}}^{\Sf_{A_{ik_i}}}(t).\]

To conclude, it now suffices to show that each ${\mathfrak{h}^*}_{A_{ij}}^{\Sf_{A_{ij}}}(t)$ is $\gamma$-effective. This is precisely the content of \Cref{subsec:cube}, see in particular \eqref{eq:HMSSS}.
\end{proof}

\section{An example} \label{sec:example}

We devote this section to a guided example illustrating the $\gamma$-effectiveness of ${\mathfrak{h}^*}^G_P(t)$ when $P$ is $\mathbf{1}$-graded and $G \subseteq \mathrm{Aut}(P)$.

Throughout this section, we consider the $\mathbf{1}$-graded poset $P=\{p_1,\ldots,p_8\}$ as in Figure~\ref{fig:D_4}. Note that, in this case, the $\mathbf{1}$-graded and the parity-graded structures of $P$ coincide.

\begin{center}
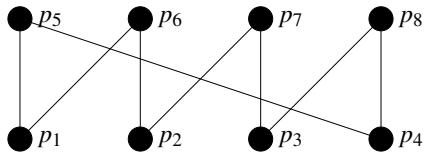
\begin{figure}[h]
\scalebox{0.8}{\begin{tikzpicture}
\foreach \x in {1, 2, 3, 4} {
    \draw[fill=black] (\x*2-2, 0) circle [radius=0.2];
    \node at (\x*2 - 1.5, 0) {$p_{\x}$}; 
}
\foreach \x in {5,6,7,8} {
    \draw[fill=black] (\x*2-10, 2) circle [radius=0.2];
    \node at (\x*2 - 9.5, 2) {$p_{\x}$}; 
}
\foreach \x in {0, 1, 2, 3} {
    \draw (\x*2, 0) -- (\x*2, 2); 
}
\draw (0,0) -- (2,2);  
\draw (2,0) -- (4,2); 
\draw (4,0) -- (6,2); 
\draw (6,0) -- (0,2); 
\end{tikzpicture}}

\caption{The poset $P$.}\label{fig:D_4}

\end{figure}
\end{center}
Let $G=D_4=\langle \sigma, \tau \mid \sigma^4=\tau^2=e, \sigma\tau\sigma=\tau\rangle$ be the dihedral group of order $8$. We record in \Cref{tab:D_4} its table of characters, and note here for later use that the character $\chi_{\mathrm{reg}}$ of the regular representation of $G$ (which arises as the representation induced on $G$ by the trivial one on $\{e\}$) is $\chi_{\mathrm{reg}} = 1 + \chi_1 + \chi_2 + \chi_3 + 2\chi_4$.

\begin{center}
\begin{table}
\begin{tabular}{l|rrrrr}
$D_4$ &$1$ &$\chi_1$ &$\chi_2$ &$\chi_3$ &$\chi_4$\\ \hline
$e$ &$1$ &$1$ &$1$ &$1$ &$2$ \\
$\{\sigma,\sigma^3\}$ &$1$ &$1$ &$-1$ &$-1$ &$0$ \\
$\sigma^2$ &$1$ &$1$ &$1$ &$1$ &$-2$ \\
$\{\tau,\tau\sigma^2\}$ &$1$ &$-1$ &$1$ &$-1$ &$0$ \\ 
$\{\tau\sigma,\tau\sigma^3\}$ &$1$ &$-1$ &$-1$ &$1$ &$0$ \\ 
\end{tabular}
\bigskip
\caption{The character table of the dihedral group $D_4$.}\label{tab:D_4}
\end{table}
\end{center}

The group $G$ acts on $P$ as follows: \vspace{-0.4cm}
\[\hspace{-1cm}
\sigma \cdot p_i=\begin{cases}p_{i+1} &\text{ for }i=1,2,3,5,6,7; \\ p_1 &\text{ for }i=4; \\ p_5 &\text{ for }i=8; \end{cases} 
\quad\text{and}\quad \tau\cdot p_i=\begin{cases}
p_i &\text{ for }i=1,3; \\
p_4 &\text{ for }i=2; \\
p_2 &\text{ for }i=4; \\
p_{i+1} &\text{ for }i=5,7; \\
p_{i-1} &\text{ for }i=6,8. 
\end{cases}
\]
\vspace{-0.4cm}
\begin{figure}[h]
\begin{minipage}[h]{0.48\linewidth}
\hspace{0.8cm}
\scalebox{0.8}{
\begin{tikzpicture}
\foreach \x in {1, 2, 3, 4} {
    \draw[fill=black] (\x*2-2, 0) circle [radius=0.2];
    \node at (\x*2 - 1.5, -0.5) {$p_{\x}$}; 
}
\foreach \x in {5,6,7,8} {
    \draw[fill=black] (\x*2-10, 2) circle [radius=0.2];
    \node at (\x*2 - 9.5, 1.5) {$p_{\x}$}; 
}
\foreach \x in {0, 1, 2, 3} {
    \draw (\x*2, 0) -- (\x*2, 2); 
}
\draw (0,0) -- (2,2);  
\draw (2,0) -- (4,2); 
\draw (4,0) -- (6,2); 
\draw (6,0) -- (0,2); 
\foreach \x in {1, 2, 3}{
\draw[->, ultra thick] (\x*2-1.7, 0) -- (\x*2-0.3,0); 
\draw[->, ultra thick] (\x*2-1.7, 2) -- (\x*2-0.3,2); 
}
\draw[<-, ultra thick] (0.2, 0.2) .. controls (3, 0.8) .. (5.8, 0.2);
\draw[<-, ultra thick] (0.2, 2.2) .. controls (3, 2.8) .. (5.8, 2.2);
\end{tikzpicture}}
\caption{The action by $\sigma$.}
\end{minipage}
\begin{minipage}[h]{0.48\linewidth}
\hspace{0.8cm}
\scalebox{0.8}{
\begin{tikzpicture}
\foreach \x in {1, 2, 3, 4} {
    \draw[fill=black] (\x*2-2, 0) circle [radius=0.2];
    \node at (\x*2 - 1.5, -0.5) {$p_{\x}$}; 
}
\foreach \x in {5,6,7,8} {
    \draw[fill=black] (\x*2-10, 2) circle [radius=0.2];
    \node at (\x*2 - 9.5, 1.5) {$p_{\x}$}; 
}
\foreach \x in {0, 1, 2, 3} {
    \draw (\x*2, 0) -- (\x*2, 2); 
}
\draw (0,0) -- (2,2);  
\draw (2,0) -- (4,2); 
\draw (4,0) -- (6,2); 
\draw (6,0) -- (0,2); 
\draw[<->, ultra thick] (2.2, 0.2) .. controls (4, 0.8) .. (5.8, 0.2);
\draw[<->, ultra thick] (0.2, 2.2) .. controls (1, 2.8) .. (1.8, 2.2); 
\draw[<->, ultra thick] (4.2, 2.2) .. controls (5, 2.8) .. (5.8, 2.2);
\end{tikzpicture}}
\caption{The action by $\tau$.}
\end{minipage}
\end{figure}

We describe in \Cref{fig:saturations} all saturations of $P$, where only $(-1)$-labelings are explicitly marked, and all saturations are grouped by $G$-representatives as in \Cref{thm:main theorem on h^*}. For example, regarding $Q_2$, ``$\times 8$'' means that the set of all saturations of $P$ contains 8 copies of the same poset with different vertex labelings, and those form one $G$-orbit. The possible vertex labelings of the $Q_2$-shape are shown in \Cref{fig:G-orbit of Q_2}.

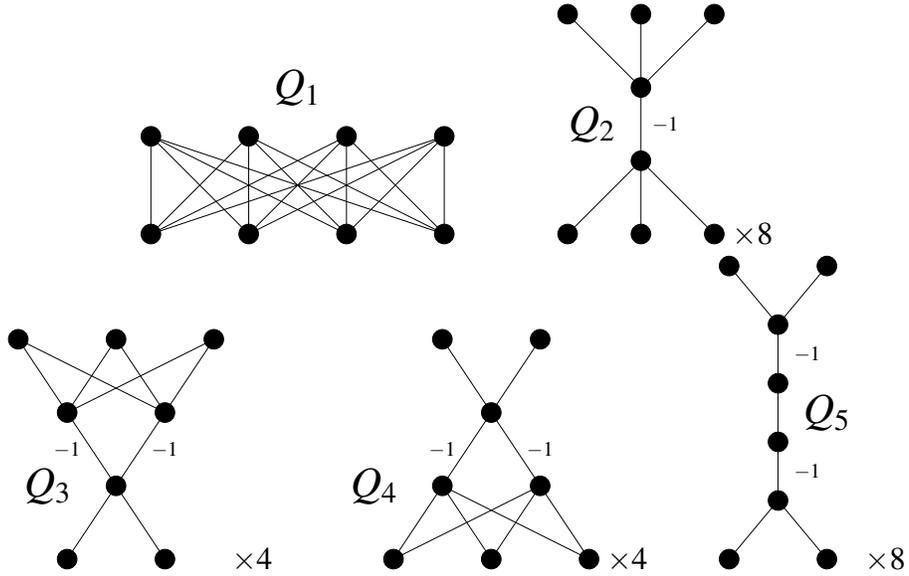
\begin{figure}[h]
\begin{center}
\scalebox{0.65}{\begin{tikzpicture}
\foreach \x in {1, 2, 3, 4} {
    \draw[fill=black] (\x*2-2, 1) circle [radius=0.2];
}
\foreach \x in {5,6,7,8} {
    \draw[fill=black] (\x*2-10, 3) circle [radius=0.2];
}
\foreach \x in {0, 1, 2, 3} {
    \draw (\x*2, 1) -- (\x*2, 3); 
}
\foreach \x in {0, 1, 2} {
    \draw (\x*2, 1) -- (\x*2+2, 3); 
    \draw (\x*2+2, 1) -- (\x*2, 3); 
}
\foreach \x in {0, 1} {
    \draw (\x*2, 1) -- (\x*2+4, 3); 
    \draw (\x*2+4, 1) -- (\x*2, 3); 
}
\draw (0,1) -- (6,3);
\draw (6,1) -- (0,3);
\node at (3,4) {{\huge $Q_1$}};
\end{tikzpicture}}
\quad\quad\quad
\scalebox{0.65}{\begin{tikzpicture}
\foreach \x in {1, 2, 3} {
    \draw[fill=black] (\x*1.5-1.5, 0) circle [radius=0.2];
    \draw[fill=black] (\x*1.5-1.5, 4.5) circle [radius=0.2];
}
\draw[fill=black] (1.5,1.5) circle [radius=0.2]; 
\draw[fill=black] (1.5,3) circle [radius=0.2]; 
\foreach \x in {0, 1, 2} {
    \draw (\x*1.5, 0) -- (1.5, 1.5); 
    \draw (1.5,3) -- (\x*1.5, 4.5); 
}
\draw (1.5,1.5) -- (1.5,3);
\node at (2, 2.25) {$-1$}; 
\node at (0.5, 2.25) {{\huge $Q_2$}};
\end{tikzpicture}} $\times 8$
\end{center}

\begin{center}
\scalebox{0.65}{\begin{tikzpicture}
\foreach \x in {1, 2} {
    \draw[fill=black] (\x*2, 0) circle [radius=0.2];
    \draw[fill=black] (\x*2, 3) circle [radius=0.2];
}
\draw[fill=black] (3,1.5) circle [radius=0.2]; 
\node at (1.6, 1.5) {{\huge $Q_3$}};
\foreach \x in {1, 2, 3} {
    \draw[fill=black] (\x*2-1, 4.5) circle [radius=0.2];
}
\foreach \x in {1, 2} {
    \draw (\x*2, 0) -- (3, 1.5); 
    \draw (\x*2, 3) -- (3, 1.5); 
}
\foreach \x in {1, 2, 3} {
    \draw (2, 3) -- (2*\x-1, 4.5); 
    \draw (4, 3) -- (2*\x-1, 4.5); 
}
\node at (2, 2.25) {$-1$}; 
\node at (4, 2.25) {$-1$}; 
\end{tikzpicture}} $\times 4$ \quad\quad
\scalebox{0.65}{\begin{tikzpicture}
\foreach \x in {1, 2} {
    \draw[fill=black] (\x*2, 4.5) circle [radius=0.2];
    \draw[fill=black] (\x*2, 1.5) circle [radius=0.2];
}
\draw[fill=black] (3,3) circle [radius=0.2]; 
\foreach \x in {1, 2, 3} {
    \draw[fill=black] (\x*2-1, 0) circle [radius=0.2];
}
\foreach \x in {1, 2} {
    \draw (\x*2, 4.5) -- (3, 3); 
    \draw (\x*2, 1.5) -- (3, 3); 
}
\node at (0.6, 1.5) {{\huge $Q_4$}};
\foreach \x in {1, 2, 3} {
    \draw (2, 1.5) -- (2*\x-1, 0); 
    \draw (4, 1.5) -- (2*\x-1, 0); 
}
\node at (2, 2.25) {$-1$}; 
\node at (4, 2.25) {$-1$}; 
\end{tikzpicture}} $\times 4$ \quad\quad
\scalebox{0.65}{\begin{tikzpicture}
\foreach \x in {1, 2} {
    \draw[fill=black] (\x*2, 0) circle [radius=0.2];
    \draw[fill=black] (\x*2, 6) circle [radius=0.2];
}
\foreach \x in {1, 2, 3, 4} {
    \draw[fill=black] (3,\x*1.2) circle [radius=0.2]; 
}
\foreach \x in {1, 2} {
    \draw (\x*2, 0) -- (3, 1.2); 
    \draw (\x*2, 6) -- (3, 4.8); 
}
\foreach \x in {1, 2, 3} {
    \draw (3, 1.2*\x) -- (3, 1.2*\x+1.2); 
}
\node at (3.6, 1.8) {$-1$}; 
\node at (3.6, 4.2) {$-1$}; 
\node at (4, 3) {{\huge $Q_5$}};
\end{tikzpicture}} $\times 8$ 
\caption{The saturations of $(P,\mathbf{1}) = (P,\varepsilon_{\mathrm{par}})$.} \label{fig:saturations}
\end{center}
\end{figure}

\begin{figure}[h]
\begin{center}
\scalebox{0.65}{
\begin{tikzpicture}
\foreach \x in {1, 2, 3} {
    \draw[fill=black] (\x*1.5-1.5, 0) circle [radius=0.2];
    \draw[fill=black] (\x*1.5-1.5, 4.5) circle [radius=0.2];
}
\draw[fill=black] (1.5,1.5) circle [radius=0.2]; 
\draw[fill=black] (1.5,3) circle [radius=0.2]; 
\foreach \x in {0, 1, 2} {
    \draw (\x*1.5, 0) -- (1.5, 1.5); 
    \draw (1.5,3) -- (\x*1.5, 4.5); 
}
\draw (1.5,1.5) -- (1.5,3);
\node at (2, 2.25) {$-1$}; 
\node at (0.5, 0) {$p_2$}; 
\node at (2, 0) {$p_3$}; 
\node at (3.5, 0) {$p_4$}; 
\node at (0.5, 4.5) {$p_5$}; 
\node at (2, 4.5) {$p_6$}; 
\node at (3.5, 4.5) {$p_8$}; 
\node at (2, 1.5) {$p_7$}; 
\node at (2, 3) {$p_1$}; 
\end{tikzpicture} \quad\quad \begin{tikzpicture}
\foreach \x in {1, 2, 3} {
    \draw[fill=black] (\x*1.5-1.5, 0) circle [radius=0.2];
    \draw[fill=black] (\x*1.5-1.5, 4.5) circle [radius=0.2];
}
\draw[fill=black] (1.5,1.5) circle [radius=0.2]; 
\draw[fill=black] (1.5,3) circle [radius=0.2]; 
\foreach \x in {0, 1, 2} {
    \draw (\x*1.5, 0) -- (1.5, 1.5); 
    \draw (1.5,3) -- (\x*1.5, 4.5); 
}
\draw (1.5,1.5) -- (1.5,3);
\node at (2, 2.25) {$-1$}; 
\node at (0.5, 0) {$p_2$}; 
\node at (2, 0) {$p_3$}; 
\node at (3.5, 0) {$p_4$}; 
\node at (0.5, 4.5) {$p_5$}; 
\node at (2, 4.5) {$p_6$}; 
\node at (3.5, 4.5) {$p_7$}; 
\node at (2, 1.5) {$p_8$}; 
\node at (2, 3) {$p_1$}; 
\end{tikzpicture} \quad\quad 
\begin{tikzpicture}
\foreach \x in {1, 2, 3} {
    \draw[fill=black] (\x*1.5-1.5, 0) circle [radius=0.2];
    \draw[fill=black] (\x*1.5-1.5, 4.5) circle [radius=0.2];
}
\draw[fill=black] (1.5,1.5) circle [radius=0.2]; 
\draw[fill=black] (1.5,3) circle [radius=0.2]; 
\foreach \x in {0, 1, 2} {
    \draw (\x*1.5, 0) -- (1.5, 1.5); 
    \draw (1.5,3) -- (\x*1.5, 4.5); 
}
\draw (1.5,1.5) -- (1.5,3);
\node at (2, 2.25) {$-1$}; 
\node at (0.5, 0) {$p_1$}; 
\node at (2, 0) {$p_3$}; 
\node at (3.5, 0) {$p_4$}; 
\node at (0.5, 4.5) {$p_6$}; 
\node at (2, 4.5) {$p_7$}; 
\node at (3.5, 4.5) {$p_8$}; 
\node at (2, 1.5) {$p_5$}; 
\node at (2, 3) {$p_2$}; 
\end{tikzpicture} \quad\quad \begin{tikzpicture}
\foreach \x in {1, 2, 3} {
    \draw[fill=black] (\x*1.5-1.5, 0) circle [radius=0.2];
    \draw[fill=black] (\x*1.5-1.5, 4.5) circle [radius=0.2];
}
\draw[fill=black] (1.5,1.5) circle [radius=0.2]; 
\draw[fill=black] (1.5,3) circle [radius=0.2]; 
\foreach \x in {0, 1, 2} {
    \draw (\x*1.5, 0) -- (1.5, 1.5); 
    \draw (1.5,3) -- (\x*1.5, 4.5); 
}
\draw (1.5,1.5) -- (1.5,3);
\node at (2, 2.25) {$-1$}; 
\node at (0.5, 0) {$p_1$}; 
\node at (2, 0) {$p_3$}; 
\node at (3.5, 0) {$p_4$}; 
\node at (0.5, 4.5) {$p_5$}; 
\node at (2, 4.5) {$p_6$}; 
\node at (3.5, 4.5) {$p_7$}; 
\node at (2, 1.5) {$p_8$}; 
\node at (2, 3) {$p_2$}; 
\end{tikzpicture}}\end{center}
\begin{center}
\scalebox{0.65}{\begin{tikzpicture}
\foreach \x in {1, 2, 3} {
    \draw[fill=black] (\x*1.5-1.5, 0) circle [radius=0.2];
    \draw[fill=black] (\x*1.5-1.5, 4.5) circle [radius=0.2];
}
\draw[fill=black] (1.5,1.5) circle [radius=0.2]; 
\draw[fill=black] (1.5,3) circle [radius=0.2]; 
\foreach \x in {0, 1, 2} {
    \draw (\x*1.5, 0) -- (1.5, 1.5); 
    \draw (1.5,3) -- (\x*1.5, 4.5); 
}
\draw (1.5,1.5) -- (1.5,3);
\node at (2, 2.25) {$-1$}; 
\node at (0.5, 0) {$p_1$}; 
\node at (2, 0) {$p_2$}; 
\node at (3.5, 0) {$p_4$}; 
\node at (0.5, 4.5) {$p_6$}; 
\node at (2, 4.5) {$p_7$}; 
\node at (3.5, 4.5) {$p_8$}; 
\node at (2, 1.5) {$p_5$}; 
\node at (2, 3) {$p_3$}; 
\end{tikzpicture} \quad\quad \begin{tikzpicture}
\foreach \x in {1, 2, 3} {
    \draw[fill=black] (\x*1.5-1.5, 0) circle [radius=0.2];
    \draw[fill=black] (\x*1.5-1.5, 4.5) circle [radius=0.2];
}
\draw[fill=black] (1.5,1.5) circle [radius=0.2]; 
\draw[fill=black] (1.5,3) circle [radius=0.2]; 
\foreach \x in {0, 1, 2} {
    \draw (\x*1.5, 0) -- (1.5, 1.5); 
    \draw (1.5,3) -- (\x*1.5, 4.5); 
}
\draw (1.5,1.5) -- (1.5,3);
\node at (2, 2.25) {$-1$}; 
\node at (0.5, 0) {$p_1$}; 
\node at (2, 0) {$p_2$}; 
\node at (3.5, 0) {$p_4$}; 
\node at (0.5, 4.5) {$p_5$}; 
\node at (2, 4.5) {$p_7$}; 
\node at (3.5, 4.5) {$p_8$}; 
\node at (2, 1.5) {$p_6$}; 
\node at (2, 3) {$p_3$}; 
\end{tikzpicture} \quad\quad
\begin{tikzpicture}
\foreach \x in {1, 2, 3} {
    \draw[fill=black] (\x*1.5-1.5, 0) circle [radius=0.2];
    \draw[fill=black] (\x*1.5-1.5, 4.5) circle [radius=0.2];
}
\draw[fill=black] (1.5,1.5) circle [radius=0.2]; 
\draw[fill=black] (1.5,3) circle [radius=0.2]; 
\foreach \x in {0, 1, 2} {
    \draw (\x*1.5, 0) -- (1.5, 1.5); 
    \draw (1.5,3) -- (\x*1.5, 4.5); 
}
\draw (1.5,1.5) -- (1.5,3);
\node at (2, 2.25) {$-1$}; 
\node at (0.5, 0) {$p_1$}; 
\node at (2, 0) {$p_2$}; 
\node at (3.5, 0) {$p_3$}; 
\node at (0.5, 4.5) {$p_5$}; 
\node at (2, 4.5) {$p_7$}; 
\node at (3.5, 4.5) {$p_8$}; 
\node at (2, 1.5) {$p_6$}; 
\node at (2, 3) {$p_4$}; 
\end{tikzpicture} \quad\quad \begin{tikzpicture}
\foreach \x in {1, 2, 3} {
    \draw[fill=black] (\x*1.5-1.5, 0) circle [radius=0.2];
    \draw[fill=black] (\x*1.5-1.5, 4.5) circle [radius=0.2];
}
\draw[fill=black] (1.5,1.5) circle [radius=0.2]; 
\draw[fill=black] (1.5,3) circle [radius=0.2]; 
\foreach \x in {0, 1, 2} {
    \draw (\x*1.5, 0) -- (1.5, 1.5); 
    \draw (1.5,3) -- (\x*1.5, 4.5); 
}
\draw (1.5,1.5) -- (1.5,3);
\node at (2, 2.25) {$-1$}; 
\node at (0.5, 0) {$p_1$}; 
\node at (2, 0) {$p_2$}; 
\node at (3.5, 0) {$p_3$}; 
\node at (0.5, 4.5) {$p_5$}; 
\node at (2, 4.5) {$p_6$}; 
\node at (3.5, 4.5) {$p_8$}; 
\node at (2, 1.5) {$p_7$}; 
\node at (2, 3) {$p_4$}; 
\end{tikzpicture}}
\caption{The $G$-orbit of the saturation $(Q_2, \varepsilon_{\mathrm{par}})$.} \label{fig:G-orbit of Q_2}
\end{center}
\end{figure}
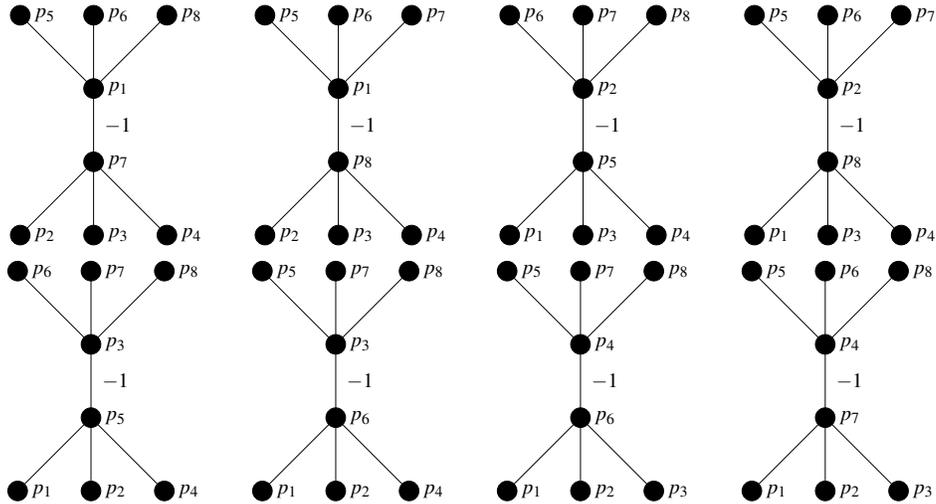

Below, we compute the $\gamma$-polynomial of the equivariant $h^*$-polynomial of $P$ by going through the summands of \eqref{eq:gamma via saturations}, which correspond to the five $G$-orbits partitioning the set of all saturations of $P$. 

\medskip

\noindent
\fbox{$Q_1$} \; Note that this class forms a single $G$-orbit: hence, by the orbit-stabilizer theorem, $\mathrm{stab}_G(Q_1,\varepsilon_{\mathrm{par}})=G$. Moreover, this saturation is of the form $A_1 \oplus_1 A_2$, where $|A_1|=|A_2|=4$, and the parity-graded structure of $Q_1$ coincides with the $\mathbf{1}$-graded one. In particular, $r_{Q_1}(\mathbf{1}) = r_P(\mathbf{1})$.

Therefore, one has that the summand corresponding to $Q_1$ in \eqref{eq:gamma via saturations} is

\begin{equation} \label{eq:Q1 example}
\mathrm{Ind}^G_G\ \mathrm{Res}^{\Sf_{A_1}\times \Sf_{A_2}}_{G}\left(t^0 \cdot {\boldsymbol\gamma}^{\Sf_{A_1}\times \Sf_{A_2}}_{A_1 \oplus_1 A_2}(t)\right) = \mathrm{Res}^{\Sf_{A_1}\times \Sf_{A_2}}_{G}\left(\boldsymbol\gamma^{\Sf_{A_1}}_{A_1}(t)\cdot \boldsymbol\gamma^{\Sf_{A_2}}_{A_2}(t)\right).
\end{equation}

By the results in \Cref{subsec:cube}, we know how to write an explicit expression for $\boldsymbol\gamma^{\Sf_{A_i}}_{A_i}(t)$; indeed, we know that \[\boldsymbol\gamma^{\Sf_{A_i}}_{A_i}(t) = 1 + \chi_{A_i,1}t,\] where $\chi_{A_i,1}$ is a suitable $\Sf_4$-character. Then the quantity in \eqref{eq:Q1 example} becomes

\[\begin{split}&\phantom{==}\!\mathrm{Res}^{\mathfrak{S}_{A_1} \times \Sf_{A_2}}_G\left((1+\chi_{A_1,1}t)(1+\chi_{A_2,1}t)\right)
\\&= 1 + \left(\mathrm{Res}^{\mathfrak{S}_{A_1} \times \Sf_{A_2}}_G(\chi_{A_1,1}+\chi_{A_2,1})\right)t + \left(\mathrm{Res}^{\mathfrak{S}_{A_1} \times \Sf_{A_2}}_G(\chi_{A_1,1}\chi_{A_2,1})\right)t^2\\&=:1+\boldsymbol\gamma_1^{(1)}t+\boldsymbol\gamma_2^{(1)}t^2.\end{split}\]

In order to get explicit expressions for $\boldsymbol\gamma_1^{(1)}$ and $\boldsymbol\gamma_2^{(1)}$, we first need to describe the characters $\chi_{A_i,1}$. We said earlier that these are ``morally'' $\Sf_4$-characters, but we actually need to be more careful, as $\Sf_4$ embeds in two different ways into the direct product $\Sf_4 \times \Sf_4$; more precisely, it embeds as $\Sf_4 \times \{e\}$ when it acts on the cube $\mathscr{O}(A_1)$, and as $\{e\} \times \Sf_4$ when it acts on the cube $\mathscr{O}(A_2)$. To remember such a difference, we will decorate the irreducible $\Sf_4$-characters in play with an ``$i$'' (e.g., $\chi^{{\tiny\ydiagram{2,2}}, i}$) if they arise from the action on $\mathscr{O}(A_i)$. By \eqref{eq:ex}, for $i \in \{1,2\}$ we get that 
\[\chi_{A_i,1} = \chi^{{\tiny\ydiagram{2,2}}, i}+2\chi^{{\tiny\ydiagram{3,1}}, i}.\]
The explicit decompositions of $\gamma_1^{(1)}$ and $\gamma_2^{(1)}$ as sums of irreducible $G$-characters are as follows: 
\[\gamma_1^{(1)}=2+3\chi_2+3\chi_3+4\chi_4 \quad\text{and} \quad \gamma_2^{(1)}=9+9\chi_1+7\chi_2+7\chi_3+16\chi_4,\]
where each $\chi_i$ is as in Table~\ref{tab:D_4}. 
We briefly explain how to get the decomposition of $\gamma_2^{(1)}$; one follows a similar process for $\gamma_1^{(1)}$. Since 
each conjugacy class of $G$ translates as one of $\Sf_4 \times \Sf_4$ via the following dictionary \begin{align*}
e\in G &\;\longleftrightarrow\; (e,e) \in \Sf_{A_1} \times \Sf_{A_2}; \\ 
\sigma, \sigma^3 \in G &\;\longleftrightarrow\; ((1 \, 2 \, 3 \, 4),(1 \, 2 \, 3 \, 4)) \in \Sf_{A_1} \times \Sf_{A_2}; \\
\sigma^2 \in G &\;\longleftrightarrow\; ((1 \, 2)(3 \, 4),(1 \, 2)(3 \, 4)) \in \Sf_{A_1} \times \Sf_{A_2}; \\
\tau, \tau\sigma^2 \in G &\;\longleftrightarrow\; ((1 \, 2),(1 \, 2)(3 \, 4)) \in \Sf_{A_1} \times \Sf_{A_2}; \\ 
\tau\sigma,\tau\sigma^3 \in G &\;\longleftrightarrow\; ((1 \, 2)(3 \, 4), (1 \, 2)) \in \Sf_{A_1} \times \Sf_{A_2}, 
\end{align*}
using \Cref{tab:S_3} we see that 
\begin{align*}
\gamma_2^{(1)}(g)&=\mathrm{Res}_G^{\Sf_{A_1} \times \Sf_{A_2}}(\chi_{A_1,1}\chi_{A_2,1})(g) \\
&=\mathrm{Res}_G^{\Sf_{A_1} \times \Sf_{A_2}}((\chi^{{\tiny\ydiagram{2,2}}, 1}+2\chi^{{\tiny\ydiagram{3,1}}, 1})(\chi^{{\tiny\ydiagram{2,2}}, 2}+2\chi^{{\tiny\ydiagram{3,1}}, 2}))(g) \\
&=\mathrm{Res}_G^{\Sf_{A_1} \times \Sf_{A_2}}(\chi^{{\tiny\ydiagram{2,2}}, 1}\chi^{{\tiny\ydiagram{2,2}}, 2}+2\chi^{{\tiny\ydiagram{3,1}}, 1}\chi^{{\tiny\ydiagram{2,2}}, 2}+2\chi^{{\tiny\ydiagram{2,2}}, 1}\chi^{{\tiny\ydiagram{3,1}}, 2}+4\chi^{{\tiny\ydiagram{3,1}}, 1}\chi^{{\tiny\ydiagram{3,1}}, 2})(g) \\
&=\begin{cases}
64 &\text{ if }g=e, \\
4 &\text{ if }g=\sigma,\sigma^3, \\
0 &\text{ if }g=\sigma^2, \\
0 &\text{ if }g=\tau,\tau\sigma^2, \\
0 &\text{ if }g=\tau\sigma,\tau\sigma^3, 
\end{cases}
\end{align*}
which by comparison with \Cref{tab:D_4} yields that $\gamma_2^{(1)}=9+9\chi_1+7\chi_2+7\chi_3+16\chi_4$.

\medskip

\noindent
\fbox{$Q_2$} \; There are $8$ copies of the same poset $(Q_2,\varepsilon_{\mathrm{par}})$, which is of the form $A_1 \oplus_1 A_2 \oplus_{-1} A_3 \oplus_1 A_4$, where $|A_1|=|A_4|=3$ and $|A_2|=|A_3|=1$. 
This implies that $\mathrm{stab}_G(Q_2, \varepsilon_{\mathrm{par}})=\{e\}$. 
In this case, $r_{Q_2}(\mathbf{1})=3$. The summand corresponding to $Q_2$ in \eqref{eq:gamma via saturations} is then
\begin{align*}
\mathrm{Ind}^G_{\{e\}}&\mathrm{Res}^{\Sf_{A_1} \times \Sf_{A_2} \times \Sf_{A_3} \times \Sf_{A_4}}_{\{e\}}\left( t^{\frac{3-1}{2}} \cdot \gamma_{A_1 \oplus_1 A_2 \oplus_1 A_3 \oplus_1 A_4}^{\Sf_{A_1} \times \Sf_{A_2} \times \Sf_{A_3} \times \Sf_{A_4}}(t)\right) \\
&=t \cdot \mathrm{Ind}^G_{\{e\}}\mathrm{Res}^{\Sf_{A_1} \times \cdots \times \Sf_{A_4}}_{\{e\}}\left( \gamma_{A_1}^{\Sf_{A_1}}(t)\gamma_{A_2}^{\Sf_{A_2}}(t)\gamma_{A_3}^{\Sf_{A_3}}(t)\gamma_{A_4}^{\Sf_{A_4}}(t)\right) \\
&=t \cdot \mathrm{Ind}^G_{\{e\}}\mathrm{Res}^{\Sf_{A_1} \times \cdots \times \Sf_{A_4}}_{\{e\}}\left((1+\chi_{A_1,1}t)\cdot 1 \cdot 1 \cdot
(1+\chi_{A_4,1}t)\right) \\
&=t \cdot \mathrm{Ind}^G_{\{e\}}(1+4t+4t^2), 
\end{align*}
where for $i \in \{1,4\}$ we get that $\chi_{A_i,1} = \chi^{{\tiny\ydiagram{2,1}}, i}$ by \eqref{eq:ex}, and thus $\mathrm{Res}^G_{\{e\}} \chi_{A_1,1} = \mathrm{Res}^G_{\{e\}} \chi_{A_4,1} = 2$. Recalling that inducing from the trivial representation on $\{e\}$ gives rise to the regular character $\chi_{\mathrm{reg}}=1+\chi_1+\chi_2+\chi_3+2\chi_4$, we conclude that 
\[
\mathrm{Ind}^G_{\{e\}}\mathrm{Res}^{\Sf_{A_1} \times \Sf_{A_2} \times \Sf_{A_3} \times \Sf_{A_4}}_{\{e\}}\left( t \cdot \gamma_{A_1 \oplus_1 A_2 \oplus_1 A_3 \oplus_1 A_4}^{\Sf_{A_1} \times \Sf_{A_2} \times \Sf_{A_3} \times \Sf_{A_4}}(t)\right) = \chi_{\mathrm{reg}} \cdot t(1+4t+4t^2). 
\]

\medskip

\noindent
\fbox{$Q_3$} \; There are $4$ copies of the poset $Q_3$, which is of the form $A_1 \oplus_1 A_2 \oplus_{-1} A_3 \oplus_1 A_4$, 
where $|A_1|=|A_3|=2$, $|A_2|=1$ and $|A_4|=3$. In this case, we have $\mathrm{stab}_G(Q_3,\varepsilon_{\mathrm{par}})=\{e,\tau\sigma\}$ or $\{e,\tau\sigma^3\}$, depending on the choice of representatives. 
Let, say, $\mathrm{stab}_G(Q_3,\varepsilon_{\mathrm{par}})=\{e,\tau\sigma\}$. Noting that $r_{Q_3}(\mathbf{1}) = 3$, we see that 
\begin{align*}
\mathrm{Ind}^G_{\{e,\tau\sigma\}}&\mathrm{Res}^{\Sf_{A_1} \times \cdots \times \Sf_{A_4}}_{\{e,\tau\sigma\}}\left(t^{\frac{3-1}{2}} \cdot \gamma^{\Sf_{A_1} \times \Sf_{A_2} \times \Sf_{A_3} \times \Sf_{A_4}}_{A_1 \oplus_1 A_2 \oplus_1 A_3 \oplus_1 A_4}(t)\right) \\
&=t\cdot \mathrm{Ind}^G_{\{e,\tau\sigma\}}\mathrm{Res}^{\Sf_{A_4}}_{\{e,\tau\sigma\}}\gamma^{\Sf_{A_4}}_{A_4}(t) \\
&=t\cdot \mathrm{Ind}^G_{\{e,\tau\sigma\}}\mathrm{Res}^{\Sf_{A_4}}_{\{e,\tau\sigma\}}(1+\chi_{A_4,1}t),\end{align*}
where $\chi_{A_4,1}=\chi^{{\tiny\ydiagram{2,1}},4}$. One can check that $\tau\sigma$ gives rise to a transposition of $\Sf_{A_4}$; hence, the restriction of $\chi^{{\tiny\ydiagram{2,1}},4}$ to $\{e,\tau\sigma\} (\cong \mathbb{Z}_2)$ equals $2$ when evaluated at $e$ and $0$ when evaluated at $\tau\sigma$, and thus $\mathrm{Res}^{\Sf_{A_4}}_{\{e,\tau\sigma\}}(\chi^{{\tiny\ydiagram{2,1}},4}) = 1 + \chi_{\mathrm{sgn}}$ (here $\chi_{\mathrm{sgn}}$ is the alternating character of $\mathbb{Z}_2$). Since $1 + \chi_{\sgn} = \mathrm{Ind}^{\{e, \tau\sigma\}}_{\{e\}}(1)$, one has that $\mathrm{Ind}^G_{\{e,\tau\sigma\}}\mathrm{Ind}^{\{e,\tau\sigma\}}_{\{e\}}(1) = \mathrm{Ind}^G_{\{e\}}(1) = \chi_{\mathrm{reg}}$. The character $\mathrm{Ind}^G_{\{e,\tau\sigma\}}(1)$ is instead the character of the permutation action of $G$ on the set of left cosets $G/\{e,\tau\sigma\}$, and one checks that this equals $1 + \chi_3 + \chi_4$. Putting all the pieces together, we obtain that
\[\begin{split}&\phantom{=}\ \ \mathrm{Ind}^G_{\{e,\tau\sigma\}}\mathrm{Res}^{\Sf_{A_1} \times \cdots \times \Sf_{A_4}}_{\{e,\tau\sigma\}}\left(t \cdot \gamma^{\Sf_{A_1} \times \Sf_{A_2} \times \Sf_{A_3} \times \Sf_{A_4}}_{A_1 \oplus_1 A_2 \oplus_1 A_3 \oplus_1 A_4}(t)\right)
\\&=t\cdot \mathrm{Ind}^G_{\{e,\tau\sigma\}}\left(\mathrm{Res}^{\Sf_{A_4}}_{\{e,\tau\sigma\}}(1+\chi^{{\tiny\ydiagram{2,1}},4} \cdot t)\right)
\\&=t\cdot \mathrm{Ind}^G_{\{e,\tau\sigma\}}\left(1+ \mathrm{Ind}^{\{e,\tau\sigma\}}_{\{e\}}(1) \cdot t\right)
\\&=t\cdot \left(\mathrm{Ind}^G_{\{e,\tau\sigma\}}(1)+\mathrm{Ind}^G_{\{e\}}(1) \cdot t\right)
\\&=t\cdot (1+\chi_3+\chi_4+\chi_{\mathrm{reg}}\cdot t).\end{split}\]

\medskip

\noindent
\fbox{$Q_4$} \; Similarly to the case of $Q_3$, we obtain that 
\[\mathrm{Ind}^G_{\{e,\tau\}}\mathrm{Res}^{\Sf_{A_1} \times \cdots \times \Sf_{A_4}}_{\{e,\tau\}}\left(t \cdot \gamma^{\Sf_{A_1} \times \cdots \times \Sf_{A_4}}_{A_1 \oplus_1 \cdots \oplus_1 A_4}(t)\right)=t\cdot(1+\chi_2+\chi_4+\chi_{\mathrm{reg}}\cdot t),\]
where $1+\chi_2+\chi_4$ is the character of the permutation action of $G$ on the set of left cosets $G/\{e,\tau\}$. 

\medskip

\noindent
\fbox{$Q_5$} \; There are $8$ copies of the same poset $Q_5$, 
which is of the form $A_1 \oplus_1 A_2 \oplus_{-1} A_3 \oplus_1 A_4 \oplus_{-1} A_5 \oplus_1 A_6$, where $|A_1|=|A_6|=2$, and $|A_2|=\cdots=|A_5|=1$. In particular, $\mathrm{stab}_G(Q_5,\varepsilon_{\mathrm{par}})=\{e\}$ and $r_{Q_5}(\mathbf{1}) = 5$. In this case, we obtain the following: 
\[\mathrm{Ind}^G_{\{e\}}\mathrm{Res}^{\Sf_{A_1} \times \cdots \times \Sf_{A_6}}_{\{e\}}\left(t^{\frac{5-1}{2}} \cdot \gamma^{\Sf_{A_1} \times \cdots \times \Sf_{A_6}}_{A_1 \oplus_1 \cdots \oplus_1 A_6}(t)\right) = t^2 \cdot \chi_{\mathrm{reg}}.\]

\medskip

\noindent
Putting together all the summands in \eqref{eq:gamma via saturations}, we conclude that 
\begin{align*}
\gamma_P^G(t)&=1+(2+3\chi_2+3\chi_3+4\chi_4)t+(9+9\chi_1+7\chi_2+7\chi_3+16\chi_4)t^2 \\
&\quad +t \cdot (\chi_{\mathrm{reg}}+4\chi_{\mathrm{reg}}t+4\chi_{\mathrm{reg}}t^2) \\
&\quad +t \cdot (1+\chi_3+\chi_4+\chi_{\mathrm{reg}}t)+t \cdot (1+\chi_2+\chi_4+\chi_{\mathrm{reg}}t)+t^2 \cdot \chi_{\mathrm{reg}} \\
&=1+(5+\chi_1+5\chi_2+5\chi_3+8\chi_4)t \\
&\quad +(16+16\chi_1+14\chi_2+14\chi_3+30\chi_4)t^2 \\
&\quad +4(1+\chi_1+\chi_2+\chi_3+2\chi_4)t^3. 
\end{align*}

\bigskip

\bibliographystyle{plain}
\bibliography{bibliography}

\end{document}